\documentclass[a4paper,12pt,reqno]{amsart}
\usepackage{amsfonts}
\usepackage{amsmath}
\usepackage{amssymb}
\usepackage[a4paper]{geometry}
\usepackage{mathrsfs}
\usepackage{hyperref}
\renewcommand\eqref[1]{(\ref{#1})} 
\numberwithin{equation}{section}
\theoremstyle{plain}
\newtheorem{thm}{Theorem}[section]
\newtheorem{prop}[thm]{Proposition}
\newtheorem{cor}[thm]{Corollary}

\theoremstyle{definition}
\newtheorem{defn}[thm]{Definition}
\newtheorem{rem}[thm]{Remark}

\begin{document}

   \title[Local Hardy and Rellich inequalities for sums of squares]
   {Local Hardy and Rellich inequalities for sums of squares of vector fields}

\author[Michael Ruzhansky]{Michael Ruzhansky}
\address{
  Michael Ruzhansky:
  \endgraf
  Department of Mathematics
  \endgraf
  Imperial College London
  \endgraf
  180 Queen's Gate, London SW7 2AZ
  \endgraf
  United Kingdom
  \endgraf
  {\it E-mail address} {\rm m.ruzhansky@imperial.ac.uk}
  }
\author[Durvudkhan Suragan]{Durvudkhan Suragan}
\address{
  Durvudkhan Suragan:
  \endgraf
  Institute of Mathematics and Mathematical Modelling
  \endgraf
  125 Pushkin str.
  \endgraf
  050010 Almaty
  \endgraf
  Kazakhstan
  \endgraf
  and
  \endgraf
  Department of Mathematics
  \endgraf
  Imperial College London
  \endgraf
  180 Queen's Gate, London SW7 2AZ
  \endgraf
  United Kingdom
  \endgraf
  {\it E-mail address} {\rm d.suragan@imperial.ac.uk}
  }

\thanks{The authors were supported in parts by the EPSRC
 grant EP/K039407/1 and by the Leverhulme Grant RPG-2014-02,
 as well as by the MESRK grant 5127/GF4.}


     \begin{abstract}
     We prove local refined versions of Hardy's and Rellich's inequalities as well as of uncertainty principles for sums of squares of vector fields on bounded sets of smooth manifolds under certain assumptions on the vector fields. We also give some explicit examples, in particular, for  sums of squares of vector fields on Euclidean spaces and for sub-Laplacians on stratified Lie groups.
     \end{abstract}
     \maketitle

\section{Introduction}

Let $M$ be a smooth manifold of dimension $n$ with a volume form $d\nu$. Let $\{X_{k}\}_{k=1}^{N}$, $N\leq n$, be a family of real vector fields on $M$. We denote by $\mathcal L$ the sum of their squares:
\begin{equation}\label{EQ:L}
\mathcal{L}:=\sum_{k=1}^{N}X_{k}^{2}.
\end{equation}
Operators in this form have been much studied in the literature.
For example, it is well-known from H\"ormander \cite{H67} that if the commutators of the vector fields $\{X_{k}\}_{k=1}^{N}$ generate the Lie algebra, the operator $\mathcal L$ is locally hypoelliptic. Such operators have been also studied under weaker conditions or without the hypoellipticity property.

Let us formulate assumptions that will be appearing in this paper. After this, we will discuss several settings when these assumptions are satisfied, most notably, on stratified Lie groups, as well as for operators on $\mathbb R^{n}$ satisfying the H\"ormander commutator condition of different steps.
Thus, the first assumption that we will be sometimes making is the following, made for a point $y\in M$:

\begin{itemize}
\item[(${\rm A}_{y}$)] For $y\in M$, assume that there is an open set $T_{y}\subset M$ containing $y$ such that the operator $-\mathcal{L}$
has a  fundamental solution in $T_{y}$, that is, there exists a function
$\Gamma_y\in C^{2}(T_{y}\setminus\{y\})$ such that
\begin{equation}
-\mathcal{L}\Gamma_y=\delta_y\;\textrm{ in }\;T_{y},
\end{equation}
where $\delta_y$ is the Dirac $\delta$-distribution at $y$.
\end{itemize}
We will also often write $\Gamma(x,y)=\Gamma_y(x)$ or just $\Gamma(x)$ if the point $y$ is fixed. The space $C^{2}$ here stands for the space of functions with continuous second derivatives with respect to $\{X_{k}\}_{k=1}^{N}$, see
Section \ref{Sec4} for more details.
Although we do not explicitly assume that $\mathcal L$ is hypoelliptic, the existence of a fundamental solution will imply it.

Sometimes we will need the following stronger assumption:
\begin{itemize}
\item[(${\rm A}_{y}^{+}$)] For $y\in M$, assume that
(${\rm A}_{y}$) holds and, moreover, we have
$$\Gamma_y(x)>0 \textrm{ in } T_{y}\setminus\{y\},
\textrm{ and }  \frac{1}{\Gamma_y}(y)=0.
$$
\end{itemize}
As $\Gamma_{y}$ blows up at $y$, $\frac{1}{\Gamma_y}$ is usually well-defined and is equal to $0$ at $y$.

The second assumption comes implicitly in the following definition.
There and everywhere in the sequel, we will be using the notation
$\langle X_{k}, d\nu\rangle$ for the duality product of $X_{k}$ with the volume form $d\nu$,
that is, since $d\nu$ is an $n$-form, $\langle X_{k}, d\nu\rangle$ is an $(n-1)$-form on $M$.

\begin{defn}\label{DEF:adomain}
We say that an open bounded set $\Omega\subset M$
is an {\em admissible domain} if
its boundary $\partial\Omega$ has no self-intersections, and if
the vector fields $\{X_{k}\}_{k=1}^{N}$ satisfy
\begin{equation}\label{astokes}
\sum_{k=1}^{N}\int_{\Omega}X_{k}f_{k} d\nu=
\sum_{k=1}^{N}\int_{\partial\Omega}f_{k} \langle X_{k},d\nu\rangle,
\end{equation}
for all $f_{k}\in C^{1}(\Omega)\bigcap C(\overline{\Omega})$, $k=1,\ldots,N$.
We also say that an admissible domain $\Omega$ is {\em strongly admissible with $y\in M$} if
(${\rm A}_{y}$) is satisfied, $\Omega\subset T_{y}$, and \eqref{astokes} holds for $f_{k}=vX_{k}\Gamma_{y}$ for all $v\in C^{1}(\Omega)\bigcap C(\overline{\Omega})$.
\end{defn}

The assumption of (strong) admissibility may look complicated at
the first sight, but in the examples below we show that it is in
fact rather natural and is satisfied in a number of natural
settings, so that actually any open bounded set with a piecewise
smooth boundary without self-intersections is strongly admissible,
see also especially Proposition \ref{stokes}. The condition that
the boundary $\partial\Omega$ has no self-intersections here also
means that $\partial\Omega$ is orientable. For brevity, we will
say that such boundaries are {\em simple}. Thus, before we proceed
any further, let us point out several important rather general
settings when the above conditions are all satisfied:

\begin{itemize}
\item[(E1)] Let $M$ be a stratified Lie group, $n\geq 3$, and let $\{X_{k}\}_{k=1}^{N}$ be left-invariant vector fields giving the first stratum of $M$.
Then for every $y\in M$ the assumption (${\rm A}_{y}^{+}$) is satisfied with $T_{y}=M$. Moreover, any open bounded set $\Omega\subset M$ with a piecewise smooth simple boundary is strongly admissible.
\item[(E2)] Let $M=\mathbb{R}^{n}$, $n\geq 3$, and let the vector fields $X_{k},\,k=1,\ldots,N$, $N\leq n$, be
 in the form
\begin{equation}\label{Xk}
X_{k}=\frac{\partial}{\partial x_{k}}+
\sum_{m=N+1}^{n}a_{k,m}(x)
\frac{\partial}{\partial x_{m}},
\end{equation}
where $a_{k,m}(x)$ are locally $C^{1,\alpha}$-regular for some $0<\alpha\leq
1$, where $C^{1,\alpha}$ stands for the space of functions with $X_{k}$-derivatives in the H\"older space $C^{\alpha}$ with respect to the
control distance defined by these vector fields\footnote{see Section \ref{Sec4} for the precise definition.}. Assume also\footnote{and this condition implies
that
$\{X_{k}\}_{k=1}^{N}$ satisfy H\"ormander's commutator condition of step two.} that
$$
\frac{\partial}{\partial x_{k}}=\sum_{1\leq i<j\leq N} \lambda_{k}^{i,j}(x)[X_{i},X_{j}]
$$
for all $k=N+1,\ldots,n$, with $\lambda_{k}^{i,j}\in L^{\infty}_{loc}(M).$
Then for any $y\in M$ the assumption (${\rm A}_{y}^{+}$) is satisfied.
Moreover, any open bounded set
$\Omega\subset M$ with a piecewise smooth simple boundary is strongly admissible.
\item[(E3)] More generally,
let $M=\mathbb{R}^{n}$, $n\geq 3$, and let the vector fields $X_{k},\,k=1,\ldots,N$, $N\leq n$, satisfy the
 H\"ormander commutator condition of step $r\geq 2$.
 Assume that all $X_{k},\,k=1,\ldots,N$,
belong to $C^{r,\alpha}(U)$ for some $0<\alpha\leq 1$
and $U\subset M$, and if $r=2$ we assume $\alpha=1$.
Then for any $y\in M$ the assumption (${\rm A}_{y}^{+}$) is satisfied.
Moreover, if $X_{k}$'s are in the form \eqref{Xk}, then
any open bounded set $\Omega\subset M$ with
a piecewise smooth simple boundary is strongly admissible.
\end{itemize}

Concerning example (E1), the validity of (${\rm A}_{y}^{+}$) for any $y$ follows from Folland \cite{Fol75} while the validity of \eqref{astokes} and the strong admissibility (as in Definition \ref{DEF:adomain}) for any
domain with piecewise smooth simple boundary was shown by the authors in \cite{Ruzhansky-Suragan:Carnot groups}.

Concerning (E2), the existence of a local fundamental solution, that is (${\rm A}_{y}$) for any $y\in M$ was shown by
Manfredini \cite{M12}, see also S{\'a}nchez-Calle \cite{S84} or Fefferman and
S{\'a}nchez-Calle \cite{FS86} for the positivity, thus assuring (${\rm A}_{y}^{+}$). The validity of \eqref{astokes} and the strong admissibility
for any domain with piecewise smooth simple boundary will be shown in Section
\ref{Sec4}, as well as some more details on the space $C^{r,\alpha}$ will be
given.

The condition \eqref{Xk} in (E2) and (E3) is not restrictive. In fact, it can be shown that by a
change of variables, any collection of linearly independent vector fields (say,
locally $C^{1,\alpha}$-regular) can be transformed to a collection of the same regularity and satisfying
condition \eqref{Xk}, see e.g. Manfredini \cite[page 975]{M12}.

Concerning (E3), the validity of (${\rm A}_{y}^{+}$) was shown by
Bramanti, Brandolini, Manfredini and Pedroni
\cite[Theorem 4.8 and Theorem 5.9]{BBMP}.
The validity of \eqref{astokes} and the strong admissibility
for any domain with piecewise smooth simple boundary will be shown in Section \ref{Sec4}.

The assumptions (${\rm A}_{y}$) or (${\rm A}_{y}^{+}$) hold also in some other settings.
If $\mathcal{L}$ is a hypoelliptic operator, then
the subject of the existence of local and global fundamental solutions for $\mathcal L$ is well-studied,
see e.g. \cite{M12, BLU2004, FS86, S84, OR}
for more general discussions.

The main aim of this paper is to obtain a local Hardy inequality on $M$
generalising but also refining the known Hardy inequalities by the inclusion of
boundary terms. In turn, this will also imply the corresponding versions of local
uncertainty principles with contributions from the boundary.

For these, we first establish the analogues of Green's first and second formulae in admissible domains $\Omega$
with suitable boundary expressions. Note that difficulties related to the existence of
characteristic points on $\partial\Omega$ do not appear in our formulations.
Moreover, unless we want to take $u=\Gamma$, we do not require the existence of
fundamental solutions.
Thus, our analogue of Green's first formula is the following:
if $v\in C^{1}(\Omega)\bigcap C(\overline{\Omega})$ and $u\in \{C^{2}(\Omega)\bigcap C^{1}(\overline{\Omega})\}\bigcup\{\Gamma\}$, then
\begin{equation} \label{g10}
\int_{\Omega}\left((\mathcal{\widetilde{\nabla}}v) u+v\mathcal{L}u\right) d\nu=\int_{\partial\Omega}v\langle \mathcal{\widetilde{\nabla} }u,d\nu\rangle,
\end{equation}
where we define
$$\widetilde{\nabla} u:=\sum_{k=1}^{N} (X_{k} u) X_{k}.$$
The analogue of Green's second formula is as follows:
if $u,v\in \{C^{2}(\Omega)\bigcap C^{1}(\overline{\Omega})\}\bigcup\{\Gamma\},$ then
\begin{equation}\label{g20}
\int_{\Omega}(u\mathcal{L}v-v\mathcal{L}u)d\nu
=\int_{\partial\Omega}(u\langle\widetilde{\nabla}  v,d\nu\rangle-v\langle \widetilde{\nabla}  u,d\nu\rangle).
\end{equation}
These formulae will be proved in Proposition \ref{greenformulae}.
As a consequence we also obtain several representation formulae for functions in $\Omega$.
Other versions of the integration by parts formulae are known, see e.g. \cite{CGL}, but
\eqref{g10} will become instrumental in our proof of the Hardy inequality.

The local Hardy inequality on $M$ will be expressed in terms of the fundamental solution $\Gamma=\Gamma_{y}$ in (${\rm A}_{y}$). Consequently, as an advantage over the known things in the setting of (E1), we do not need to assume that $\Gamma=C d^{2-Q}$ for a quasi-distance (or Carnot-Carath{\'e}odory distance) $d$.
If we fix some $y\in M$ and
the corresponding $T_{y}$ and $\Gamma_y$, we may just write $\Gamma$ for brevity, if the context is clear.
Thus, we show that for $\alpha\in \mathbb{R}$, $\alpha>2-\beta$ and $\beta>2$, in any strongly admissible domain $\Omega\subset T_{y}$
we have the inequality
\begin{multline}\label{iLH20}
\qquad \qquad \int_{\Omega}\Gamma^{\frac{\alpha}{2-\beta}}
|\nabla_{X} u|^{2} \,d\nu\geq
\left(\frac{\beta+\alpha-2}{2}\right)^{2}\int_{\Omega}
\Gamma^{\frac{\alpha-2}{2-\beta}} |\nabla_{X} \Gamma^{\frac{1}{2-\beta}}|^{2}
|u|^{2}\,d\nu\\+
\frac{\beta+\alpha-2}{2(\beta-2)}\int_{\partial\Omega}\Gamma^{\frac{\alpha}{2-\beta}-1}|u|^{2}
\langle\widetilde{\nabla}\Gamma,
d\nu\rangle,
\end{multline}
for all $u\in C^{1}(\Omega)\bigcap C(\overline{\Omega})$,
where we denote
$$\nabla_{X}=(X_{1},\ldots,X_{N}).$$
Consequently, this implies local versions of uncertainty principles on $M$ that will be given in Corollary \ref{Luncertainty}.

One can readily see that the inequality \eqref{iLH20} extends the classical Hardy inequality.
Indeed, in the case of ${M}=\mathbb R^{n}$ and $X_{k}=
\frac{\partial}{\partial x_{k}},\,k=1,\ldots,n$, the inequality \eqref{iLH20} recovers the classical Hardy inequality:
taking $\alpha=0$ and $\beta=n\geq 3$, the fundamental solution for the Laplacian is given by $\Gamma(x)=C_n|x|^{2-n}$ for some
constant $C_{n}$ and $|x|$ the Euclidean norm, so that \eqref{iLH20}
reduces to the classical Hardy inequality
\begin{equation}\label{HRn}
\int_{\mathbb{R}^{n}}|\nabla u(x)|^{2}dx\geq\left(\frac{n-2}{2}\right)^{2} \int_{\mathbb{R}^{n}}\frac{|u(x)|^{2}}{|x|^{2}}dx,\quad n\geq 3,
\end{equation}
where $\nabla$ is the standard gradient in $\mathbb{R}^{n}$, $u\in C_{0}^{\infty}(\mathbb{R}^{n}\backslash \{0\})$,
and the constant $\left(\frac{n-2}{2}\right)^{2}$ is known to be sharp. The constant $C_n$ does not enter
\eqref{HRn} due to the scaling invariance of the inequality \eqref{iLH20} with respect to the multiplication of $\Gamma$ by positive constants.
Hardy type inequalities have been intensively studied, see e.g. Davies and Hinz \cite{Davies-Hinz}, and Davies \cite{Davies} for a review and their applications. We also refer to more recent paper of Hoffmann-Ostenhof and Laptev \cite{Laptev15} on this subject (see also \cite{HHLT-Hardy-many-particles} and \cite{EKL:Hardy-p-Lap}) and to further references therein.
Certain Hardy and Rellich inequalities for sums of squares have been considered by
Grillo \cite{Grillo:Hardy-Rellich-PA-2003}, compared to which our results provide refinements from several points of view.

If $M$ is a homogeneous Carnot group or a stratified group, so that we are in the setting of (E1),
and $X_k$'s are the vectors from the first stratum, \eqref{iLH20} reduces to versions obtained by the authors in
\cite{Ruzhansky-Suragan:Carnot groups} using
the $\mathcal L$-gauge  $d$: taking $\alpha=0$, $\beta=Q\geq 3$,
and $d(x)=\Gamma(x,0)^{\frac{1}{2-Q}}$,
where $Q$ is the homogeneous dimension of the group, we get
\begin{equation}\label{HRC}
 \int_{\Omega}
|\nabla_{X} u|^{2} \,d\nu\geq
\left(\frac{Q-2}{2}\right)^{2}\int_{\Omega}
\frac{|\nabla_{X} d|^{2}}{d^2}
|u|^{2}\,d\nu\\+
\frac{1}{2}\int_{\partial\Omega} d^{Q-2} |u|^{2}
\langle\widetilde{\nabla}d^{2-Q},
d\nu\rangle,
\end{equation}
with the sharp constant $\left(\frac{Q-2}{2}\right)^{2}$,
as well as its weighted versions. Without the second
(boundary) term this is known, for example on the Heisenberg group \cite{GL} for a particular choice of $d(x)$, or on
Carnot groups \cite{GolKom}, which inspired our proof. We refer to these papers as well as to
\cite{GZ} for other references on this subject,
and to \cite{BCG} and \cite{BFG} for Besov space versions of Hardy inequalities on the
Heisenberg group and on graded groups, respectively. Certain boundary value considerations on the Heisenberg group also
appeared in \cite{Ruzhansky-Suragan:Kohn-Laplacian}.

The inequality \eqref{iLH20} can be thought of as a refinement of the usual Hardy
inequality from the point of view of the boundary term since this boundary term in
\eqref{iLH20} can be positive (see  \cite[Section 7]{Ruzhansky-Suragan:Carnot
groups}), thus refining the versions of the Hardy inequality when $u$ is assumed
to be compactly supported in $\Omega$ and, therefore, this boundary term is not
present. As a consequence, this also brings the corresponding refinements to
the uncertainty principles on $M$. We call these inequalities local due to the
presence of a contribution from the boundary.

\medskip
Let $y\in M$ be such that {\rm (${\rm A}_{y}^{+}$)} holds with the fundamental solution $\Gamma=\Gamma_{y}$ in $T_{y}$.
Let $\Omega\subset T_{y}$ be a strongly admissible domain,
$\alpha\in \mathbb{R}$,\ $\beta>\alpha>4-\beta$, $\beta>2$ and $R\geq\, e\, {\rm sup}_{\Omega}\Gamma^{\frac{1}{2-\beta}}$. Let
$u\in C^{2}(\Omega)\bigcap C^{1}(\overline{\Omega})$.
Then we prove the following generalised local Rellich
inequalities (Theorem \ref{LRellich}):
\begin{multline}\label{EQ:R1}
\qquad \qquad \int_{\Omega}\frac{\Gamma^{\frac{\alpha}{2-\beta}}}
{|\nabla_{X}\Gamma^{\frac{1}{2-\beta}}|^{2}}
|\mathcal{L}u|^{2}d\nu\geq \frac{(\beta+\alpha-4)^{2}(\beta-\alpha)^{2}}{16}\int_{\Omega}
\Gamma^{\frac{\alpha-4}{2-\beta}} |\nabla_{X} \Gamma^{\frac{1}{2-\beta}}|^{2}
|u|^{2}\,d\nu
\\+
\frac{(\beta+\alpha-4)^{2}(\beta-\alpha)}{4(\beta-2)}\int_{\partial\Omega}
\Gamma^{\frac{\alpha-2}{2-\beta}-1}|u|^{2}
\langle\widetilde{\nabla}\Gamma,
d\nu\rangle+\frac{(\beta+\alpha-4)(\beta-\alpha)}{4}\mathcal{C}(u)
\end{multline}
and
\begin{multline}\label{EQ:R2}
\qquad \qquad \int_{\Omega}\frac{\Gamma^{\frac{\alpha}{2-\beta}}}
{|\nabla_{X}\Gamma^{\frac{1}{2-\beta}}|^{2}}
|\mathcal{L}u|^{2}d\nu\geq \frac{(\beta+\alpha-4)^{2}(\beta-\alpha)^{2}}{16}\int_{\Omega}
\Gamma^{\frac{\alpha-4}{2-\beta}} |\nabla_{X} \Gamma^{\frac{1}{2-\beta}}|^{2}
|u|^{2}\,d\nu\\+\frac{(\beta+\alpha-4)(\beta-\alpha)}{8}\int_{\Omega}\Gamma^{\frac{\alpha-4}{2-\beta}}
|\nabla_{X} \Gamma^{\frac{1}{2-\beta}}|^{2}\,
\left({\ln}\frac{R}{\Gamma^{\frac{1}{2-\beta}}}\right)^{-2}|u|^{2}d\nu
\\+\frac{(\beta+\alpha-4)(\beta-\alpha)}{4(\beta-2)}\int_{\partial\Omega}
\Gamma^{\frac{\alpha-2}{2-\beta}-1}\left({\ln}\frac{R}{\Gamma^{\frac{1}{2-\beta}}}\right)^{-1}|u|^{2}
\langle\widetilde{\nabla}\Gamma,
d\nu\rangle
\\+
\frac{(\beta+\alpha-4)^{2}(\beta-\alpha)}{4(\beta-2)}\int_{\partial\Omega}
\Gamma^{\frac{\alpha-2}{2-\beta}-1}|u|^{2}
\langle\widetilde{\nabla}\Gamma,
d\nu\rangle+\frac{(\beta+\alpha-4)(\beta-\alpha)}{4}\mathcal{C}(u),
\end{multline}
where $\nabla_{X}=(X_{1},\ldots,X_{N})$ and
$$\mathcal{C}(u):=\frac{\alpha-2}{2-\beta}\int_{\partial\Omega}
u^{2}\Gamma^{\frac{\alpha-2}{2-\beta}-1}\langle\widetilde{\nabla}\Gamma,
d\nu\rangle-2\int_{\partial\Omega}
\Gamma^{\frac{\alpha-2}{2-\beta}}u\langle\widetilde{\nabla}u,
d\nu\rangle.$$

For $\beta>\alpha>\frac{8-\beta}{3}$, $\beta>2$ and $R\geq\, e\, {\rm sup}_{\Omega}\Gamma^{\frac{1}{2-\beta}}$, we prove the following further generalised local Rellich
inequalities (Theorem \ref{2LRellich}):
\begin{multline}\label{EQ:R3}
\qquad \qquad \int_{\Omega}\frac{\Gamma^{\frac{\alpha}{2-\beta}}}
{|\nabla_{X}\Gamma^{\frac{1}{2-\beta}}|^{2}}|\mathcal{L}u|^{2}d\nu\geq
\frac{(\beta-\alpha)^{2}}{4}\int_{\Omega}\Gamma^{\frac{\alpha-2}{2-\beta}}|\nabla_{X} u|^{2}d\nu
\\+\frac{(\beta+3\alpha-8)(\beta+\alpha-4)(\beta-\alpha)}{8(\beta-2)}\int_{\partial\Omega}
\Gamma^{\frac{\alpha-2}{2-\beta}-1}|u|^{2}
\langle\widetilde{\nabla}\Gamma,
d\nu\rangle\\+
\frac{(\beta+\alpha-4)(\beta-\alpha)}{4}\mathcal{C}(u),
\end{multline}
and
\begin{multline}\label{EQ:R4}
\qquad \qquad \int_{\Omega}\frac{\Gamma^{\frac{\alpha}{2-\beta}}}
{|\nabla_{X}\Gamma^{\frac{1}{2-\beta}}|^{2}}
|\mathcal{L}u|^{2}d\nu\geq \frac{(\beta-\alpha)^{2}}{4}\int_{\Omega}\Gamma^{\frac{\alpha-2}{2-\beta}}|\nabla_{X} u|^{2}d\nu\\+\frac{(\beta+3\alpha-8)(\beta-\alpha)}{16}\int_{\Omega}\Gamma^{\frac{\alpha-4}{2-\beta}}
|\nabla_{X} \Gamma^{\frac{1}{2-\beta}}|^{2}\,
\left({\ln}\frac{R}{\Gamma^{\frac{1}{2-\beta}}}\right)^{-2}|u|^{2}d\nu
\\
+\frac{(\beta+3\alpha-8)(\beta-\alpha)}
{8(\beta-2)}\int_{\partial\Omega}
\Gamma^{\frac{\alpha-2}{2-\beta}-1}\left({\ln}\frac{R}{\Gamma^{\frac{1}{2-\beta}}}\right)^{-1}|u|^{2}
\langle\widetilde{\nabla}\Gamma,
d\nu\rangle
\\+
\frac{(\beta+3\alpha-8)(\beta+\alpha-4)(\beta-\alpha)}{8(\beta-2)}\int_{\partial\Omega}
\Gamma^{\frac{\alpha-2}{2-\beta}-1}|u|^{2}
\langle\widetilde{\nabla}\Gamma,
d\nu\rangle+\frac{(\beta+\alpha-4)(\beta-\alpha)}{4}\mathcal{C}(u),
\end{multline}
where $\nabla_{X}=(X_{1},\ldots,X_{N})$ and
$$\mathcal{C}(u):=\frac{\alpha-2}{2-\beta}\int_{\partial\Omega}
u^{2}\Gamma^{\frac{\alpha-2}{2-\beta}-1}\langle\widetilde{\nabla}\Gamma,
d\nu\rangle-2\int_{\partial\Omega}
\Gamma^{\frac{\alpha-2}{2-\beta}}u\langle\widetilde{\nabla}u,
d\nu\rangle.$$

The classical result by Rellich appearing at the 1954 ICM in Amsterdam
\cite{Rellich:ineq-1956}
stated the inequality
\begin{equation}\label{EQ:Rellich1}
\left\|\frac{f}{|x|^{2}}\right\|_{L^{2}(\mathbb{R}^{n})}\leq \frac{4}{n(n-4)}\|\Delta f\|_{L^{2}(\mathbb{R}^{n})},\quad n\geq 5.
\end{equation}
We refer e.g. to Davies and Hinz
\cite{Davies-Hinz} for history and further extensions,
including the derivation of sharp constants, and to
\cite{Kombe:Rellich-Carnot-2010} and \cite{Lian:Rellich} for the corresponding results for the
sub-Laplacian on homogeneous Carnot groups. Inequalities \eqref{EQ:R1}-\eqref{EQ:R4} provide their refinement and extension, with respect to further interior and boundary terms and weights, with explicit formulae for the appearing constants.

The plan of this paper is as follows. In Section \ref{SEC:2} we derive versions of Green's first and second formulae, and give some of their consequences. In Section \ref{Sec3} we prove a local Hardy inequality and a local uncertainty principle. In Section \ref{SecR} local Rellich
inequalities are studied.
In Section \ref{Sec4} we give examples, in particular showing statements (E1), (E2) and (E3).

\section{Green's formulae for sums of squares and  consequences}
\label{SEC:2}

We have the following analogue of Green's formulae which will be instrumental in the proofs of Hardy and Rellich inequalities.

\begin{prop}[Green's formulae]\label{greenformulae}
Let $\Omega\subset M$ be an admissible domain. Let $v\in C^{1}(\Omega)\bigcap C(\overline{\Omega})$ and $u\in C^{2}(\Omega)\bigcap C^{1}(\overline{\Omega})$. Then
the following analogue of Green's first formula holds:
\begin{equation} \label{g1}
\int_{\Omega}\left((\mathcal{\widetilde{\nabla}}v) u+v\mathcal{L}u\right) d\nu=\int_{\partial\Omega}v\langle \mathcal{\widetilde{\nabla} }u,d\nu\rangle,
\end{equation}
where
\begin{equation} \label{gt}
\mathcal{\widetilde{\nabla} }u=\sum_{k=1}^{N}\left(X_{k}u\right)X_{k}.
\end{equation}
If $u,v\in C^{2}(\Omega)\bigcap C^{1}(\overline{\Omega}),$ then
the following analogue of Green's second formula holds:
\begin{equation}\label{g2}
\int_{\Omega}(u\mathcal{L}v-v\mathcal{L}u)d\nu
=\int_{\partial\Omega}(u\langle\widetilde{\nabla}  v,d\nu\rangle-v\langle \widetilde{\nabla}  u,d\nu\rangle).
\end{equation}
Moreover, if $\Omega$ is strongly admissible, we can put $u=\Gamma$ in \eqref{g1}, and
$u=\Gamma$ or $v=\Gamma$ in \eqref{g2}.
\end{prop}
With the notation \eqref{gt}, for functions $u$ and $v$ we can also write
$$\left(\mathcal{\widetilde{\nabla} }v\right) u=\mathcal{\widetilde{\nabla} }v u= \sum_{k=1}^{N}\left(X_{k}v\right)\left(X_{k}u\right)=\sum_{k=1}^{N}
X_{k}v X_{k}u,$$
so that
$$(\mathcal{\widetilde{\nabla} }v) u=(\mathcal{\widetilde{\nabla} }u) v $$
is a scalar.

\begin{proof}[Proof of Proposition \ref{greenformulae}]
Taking $f_{k}=vX_{k}u,$ we get

$$\sum_{k=1}^{N}X_{k}f_{k}=(\mathcal{\widetilde{\nabla} }v) u+v\mathcal{L}u.$$
By the admissibility of $\Omega$, using \eqref{astokes} we obtain
\begin{multline*}
\int_{\Omega}\left(\mathcal{\widetilde{\nabla} }v u+v\mathcal{L}u\right) d\nu$$$$=\int_{\Omega}
\sum_{k=1}^{N}X_{k}f_{k}d\nu \\
=\int_{\partial\Omega}\sum_{k=1}^{N}\langle f_{k}X_{k},d\nu\rangle
=\int_{\partial\Omega}\sum_{k=1}^{N}\langle vX_{k}uX_{k},d\nu\rangle
=\int_{\partial\Omega}v\langle \mathcal{\widetilde{\nabla} }u,d\nu\rangle,
\end{multline*}
completing the proof of Green's first formula \eqref{g1}.
Rewriting \eqref{g1} we have
$$\int_{\Omega}\left((\mathcal{\widetilde{\nabla} }u) v+u\mathcal{L}v\right)d\nu=\int_{\partial\Omega}u\langle \mathcal{\widetilde{\nabla} }v,d\nu\rangle,$$
$$\int_{\Omega}\left((\mathcal{\widetilde{\nabla} }v) u+v\mathcal{L}u\right)d\nu=\int_{\partial\Omega}v\langle \mathcal{\widetilde{\nabla} }u,d\nu\rangle.$$
By subtracting the second identity from the first one
and using $(\mathcal{\widetilde{\nabla} }u) v=(\mathcal{\widetilde{\nabla} }v) u$
we obtain Green's second formula \eqref{g2}.
If $\Omega$ is strongly admissible, we can put $\Gamma$ for $u$ or $v$ as stated since
\eqref{astokes} holds in these cases as well.
\end{proof}

\begin{rem}
It is important that the Green formulae are valid for the fundamental solution
$\Gamma$. In the classical (Euclidean) case, the Green formulae are valid for
the fundamental solution of the Laplacian and this fact can be showed simply. However,
in general, it is not trivial.
\end{rem}

When $v=1$, Proposition \ref{greenformulae} readily implies the following
analogue of Gauss' mean value type formulae:
\begin{cor}\label{COR:v0}
For any admissible domain $\Omega\subset M$, we have
$$\mathcal{L}u\geq 0 \textrm{ in } \Omega \; \Longrightarrow \; \int_{\partial\Omega}\langle \mathcal{\widetilde{\nabla} }u,d\nu\rangle\geq0$$
and
$$\mathcal{L}u\leq 0 \textrm{ in } \Omega \; \Longrightarrow \; \int_{\partial\Omega}\langle \mathcal{\widetilde{\nabla} }u,d\nu\rangle\leq 0.$$
Consequently, we also have
$$\mathcal{L}u= 0 \textrm{ in } \Omega \; \Longrightarrow \; \int_{\partial\Omega}\langle \mathcal{\widetilde{\nabla} }u,d\nu\rangle=0.$$
\end{cor}

Also, for $x\in\Omega$, taking $v=1$ and $u(y)=\Gamma(x,y)$ in \eqref{g1} we obtain:
\begin{cor}\label{COR:v1}
If $\Omega\subset M$ is a strongly admissible domain
such that $\Omega\subset T_{y}$ for all $y\in\Omega$.
Let $x\in\Omega$. Then we have
$$\int_{\partial\Omega}\langle \mathcal{\widetilde{\nabla} }\Gamma(x,y),d\nu(y)\rangle=-1,$$
where $\mathcal{\widetilde{\nabla} }\Gamma(x,y)=\mathcal{\widetilde{\nabla} }_y\Gamma(x,y)$ refers to the notation
\eqref{gt} with derivatives taken with respect to the variable $y$.
\end{cor}

The assumption of $\Omega\subset T_{y}$ for all $y\in\Omega$ above just assures that the family of $\Gamma_{y}$ is defined over $y\in\Omega$.

Assuming the conditions of Corollary \ref{COR:v1},
and putting the fundamental solution $\Gamma$ instead of $v$ in \eqref{g2} we get the following representation formulae that have applications in various boundary value problems but are also of importance on their own.
\begin{itemize}

\item Let $u\in C^{2}(\Omega)\bigcap C^{1}(\overline{\Omega})$. Then for $x\in\Omega$ we have
\begin{multline}\label{rep}
u(x)=-\int_{\Omega}\Gamma(x,y)\mathcal{L}u(y)d\nu(y)\\ -
\int_{\partial\Omega}u(y) \langle\mathcal{\widetilde{\nabla}
}\Gamma(x,y),d\nu(y)\rangle+
\int_{\partial\Omega}\Gamma(x,y)\langle\mathcal{\widetilde{\nabla}
}u(y),d\nu(y)\rangle.
\end{multline}
\item Let $u\in C^{2}(\Omega)\bigcap C^{1}(\overline{\Omega})$ and $\mathcal{L}u=0$ on $\Omega$, then for $x\in\Omega$ we have
 \begin{equation}\label{rep}
u(x)=-\int_{\partial\Omega}u(y) \langle\mathcal{\widetilde{\nabla}
}\Gamma(x,y),d\nu(y)\rangle+ \int_{\partial\Omega}\Gamma(x,y)
\langle\mathcal{\widetilde{\nabla}}u(y),d\nu(y)\rangle.
\end{equation}
\item
Let $u\in C^{2}(\Omega)\bigcap C^{1}(\overline{\Omega})$ and
\begin{equation}\label{D1}
u(x)=0,\,\, x\in\partial\Omega,
\end{equation}
then
\begin{equation}
u(x)=-\int_{\Omega}\Gamma(x,y)\mathcal{L}u(y)d\nu(y)+
\int_{\partial\Omega}\Gamma(x,y)\langle\mathcal{\widetilde{\nabla}
}u(y),d\nu(y)\rangle.
\end{equation}
\item
Let $u\in C^{2}(\Omega)\bigcap C^{1}(\overline{\Omega})$ and
\begin{equation}\label{n1}
\sum_{j=1}^{N}X_{j}u\langle X_{j} ,d\nu\rangle=0 \;\textrm{ on }\; \partial\Omega,
\end{equation}
then
 \begin{equation}\label{rep}
u(x)=-\int_{\Omega}\Gamma(x,y)\mathcal{L}u(y)d\nu(y)-\int_{\partial\Omega}u(y)
\langle\mathcal{\widetilde{\nabla}}\Gamma(x,y),d\nu(y)\rangle.
\end{equation}
\end{itemize}

\section{Local Hardy inequalities and uncertainty principles}
\label{Sec3}

We now present local refined versions of the Hardy inequality with an additional boundary term
on the right hand side.
The proof of Theorem \ref{LHardy} relies on (coordinate free) Green's first formula
that we obtained in Proposition \ref{greenformulae}.

\begin{thm}\label{LHardy} Let $y\in M$ be such that {\rm (${\rm A}_{y}^{+}$)} holds with the fundamental solution $\Gamma=\Gamma_{y}$ in $T_{y}$.
Let $\Omega\subset T_{y}$ be a strongly admissible domain, $y\not\in\partial\Omega,$
$\alpha\in \mathbb{R}$,\ $\alpha>2-\beta$, $\beta>2$ and $R\geq\, e\, {\rm sup}_{\Omega}\Gamma^{\frac{1}{2-\beta}}$. Let
$u\in C^{1}(\Omega)\bigcap C(\overline{\Omega})$.
 Then the following generalised local Hardy
inequalities are valid:
\begin{multline}\label{LH2a}
\qquad \qquad \int_{\Omega}\Gamma^{\frac{\alpha}{2-\beta}}
|\nabla_{X} u|^{2} \,d\nu\geq
\left(\frac{\beta+\alpha-2}{2}\right)^{2}\int_{\Omega}
\Gamma^{\frac{\alpha-2}{2-\beta}} |\nabla_{X} \Gamma^{\frac{1}{2-\beta}}|^{2}
|u|^{2}\,d\nu\\+
\frac{\beta+\alpha-2}{2(\beta-2)}\int_{\partial\Omega}
\Gamma^{\frac{\alpha}{2-\beta}-1}|u|^{2}
\langle\widetilde{\nabla}\Gamma,
d\nu\rangle,
\end{multline}
and its further refinement
\begin{multline}\label{LH2}
\qquad \qquad \int_{\Omega}\Gamma^{\frac{\alpha}{2-\beta}}
|\nabla_{X} u|^{2} \,d\nu\geq\\
\left(\frac{\beta+\alpha-2}{2}\right)^{2}\int_{\Omega}
\Gamma^{\frac{\alpha-2}{2-\beta}} |\nabla_{X} \Gamma^{\frac{1}{2-\beta}}|^{2}
|u|^{2}\,d\nu+\frac{1}{4}\int_{\Omega}\Gamma^{\frac{\alpha-2}{2-\beta}}
|\nabla_{X} \Gamma^{\frac{1}{2-\beta}}|^{2}\,
\left({\ln}\frac{R}{\Gamma^{\frac{1}{2-\beta}}}\right)^{-2}|u|^{2}d\nu
\\
+\frac{1}{2(\beta-2)}\int_{\partial\Omega}
\Gamma^{\frac{\alpha}{2-\beta}-1}\left({\ln}\frac{R}{\Gamma^{\frac{1}{2-\beta}}}\right)^{-1}|u|^{2}
\langle\widetilde{\nabla}\Gamma,
d\nu\rangle+
\frac{\beta+\alpha-2}{2(\beta-2)}\int_{\partial\Omega}
\Gamma^{\frac{\alpha}{2-\beta}-1}|u|^{2}
\langle\widetilde{\nabla}\Gamma,
d\nu\rangle,
\end{multline}
where $\nabla_{X}=(X_{1},\ldots,X_{N})$.
\end{thm}

Here and in the sequel we formulate two versions of the appearing inequalities. For both Hardy and Rellich inequalities, the first one is formulated as an improvement of the classical inequality by inclusion of a boundary term - it reduces to the `classical' one for functions $u$ vanishing on the boundary $\partial\Omega$. The second inequality in each theorem provides for a further refinement by including further positive interior terms and well as further boundary terms.

In \eqref{LH2a}  the boundary term can be positive (see
 \cite[Remark
7.2]{Ruzhansky-Suragan:Carnot groups} for a discussion and examples):
\begin{equation}\label{EQ:btn}
\frac{\beta+\alpha-2}{2(\beta-2)}\int_{\partial\Omega}\Gamma^{\frac{\alpha}{2-\beta}-1}|u|^{2}
\langle\widetilde{\nabla}\Gamma,
d\nu\rangle\geq 0,
\end{equation}
for some $u$, hence it can be referred as a local Hardy inequality. Moreover, for ${\rm supp}\, u\subset\Omega$ the inequality \eqref{LH2} gives a refinement to the local Hardy inequality (cf. discussions on stratified Lie groups, for example, \cite[Theorem 3.4]{Kombe:Rellich-Carnot-2010}).

\begin{proof}[Proof of Theorem \ref{LHardy}]
Proof of \eqref{LH2a}. Without loss of generality we can assume that $u$ is real-valued.
In this case, recalling that $(\widetilde{\nabla}u)u=\sum_{k=1}^{N} (X_{k}u)X_{k}u=|\nabla_{X}u|^{2}$, inequality \eqref{LH2a}
reduces to
\begin{multline}\label{LHa}
\qquad \qquad \int_{\Omega}\Gamma^{\frac{\alpha}{2-\beta}}
(\widetilde{\nabla}u)u\,d\nu\geq
\left(\frac{\beta+\alpha-2}{2}\right)^{2}\int_{\Omega}
\Gamma^{\frac{\alpha-2}{2-\beta}}
(\widetilde{\nabla}\Gamma^{\frac{1}{2-\beta}})\Gamma^{\frac{1}{2-\beta}}
u^{2}\,d\nu\\+
\frac{\beta+\alpha-2}{2(\beta-2)}\int_{\partial\Omega}
\Gamma^{\frac{\alpha}{2-\beta}-1}u^{2}
\langle\widetilde{\nabla}\Gamma,
d\nu\rangle,
\end{multline}
which we will now prove. Setting $u=d^{\gamma}q$ for some
real-valued functions $d>0$, $q$, and a constant $\gamma\not=0$ to
be chosen later, we have
$$(\widetilde{\nabla}u)u=
(\widetilde{\nabla}d^{\gamma}q)d^{\gamma}q=\sum_{k=1}^{N}X_{k}(d^{\gamma}q)X_{k}(d^{\gamma}q)$$
$$=\gamma^{2}d^{2\gamma-2}\sum_{k=1}^{N}(X_{k}d)^{2}q^{2}+2\gamma d^{2\gamma-1}q\sum_{k=1}^{N}X_{k}d\,X_{k}q+d^{2\gamma}
\sum_{k=1}^{N}(X_{k}q)^{2}$$
$$=\gamma^{2}d^{2\gamma-2}((\widetilde{\nabla}d)d)q^{2}+2\gamma d^{2\gamma-1}q(\widetilde{\nabla}d)q+d^{2\gamma}
(\widetilde{\nabla}q)q.$$
Multiplying both sides of the above equality by $d^{\alpha}$ and applying Green's first formula (see Proposition \ref{greenformulae})
to the second term in the last line we observe that
\begin{multline*}
2\gamma \int_\Omega d^{\alpha+2\gamma-1}q(\widetilde{\nabla}d)q d\nu=
\frac{\gamma}{\alpha+2\gamma}\int_\Omega (\widetilde{\nabla}d^{\alpha+2\gamma})q^2 d\nu =
\frac{\gamma}{\alpha+2\gamma}\int_\Omega (\widetilde{\nabla}q ^2) d^{\alpha+2\gamma}d\nu \\
=-\frac{\gamma}{\alpha+2\gamma}\int_\Omega q^2 \mathcal{L} d^{\alpha+2\gamma}d\nu+
\frac{\gamma}{\alpha+2\gamma}\int_{\partial \Omega} q^2 \langle \widetilde{\nabla} d^{\alpha+2\gamma},d\nu\rangle,
\end{multline*}
where we note that later on we will choose $\gamma$ so that
$d^{\alpha+2\gamma}=\Gamma$, and so Proposition \ref{greenformulae} is applicable.
Consequently, we get
\begin{multline}\label{H1a}
\int_{\Omega}d^{\alpha}(\widetilde{\nabla}u)ud\nu=
\gamma^{2}
\int_{\Omega}d^{\alpha+2\gamma-2}((\widetilde{\nabla}d)d)\,q^{2}d\nu
+ \frac{\gamma}{\alpha+2\gamma}\int_{\Omega}(\widetilde{\nabla}d^{\alpha+2\gamma})q^{2}
d\nu\\+\int_{\Omega}d^{\alpha+2\gamma}(\widetilde{\nabla}q)qd\nu
 =
\gamma^{2}
\int_{\Omega}d^{\alpha+2\gamma-2}((\widetilde{\nabla}d)d)\,q^{2}d\nu+
\frac{\gamma}{\alpha+2\gamma}\int_{\partial\Omega}q^{2}
\langle\widetilde{\nabla}d^{\alpha+2\gamma},
d\nu\rangle
\\-\frac{\gamma}{\alpha+2\gamma}\int_{\Omega}q^{2}\mathcal{L}d^{\alpha+2\gamma}
d\nu+\int_{\Omega}d^{\alpha+2\gamma}(\widetilde{\nabla}q)qd\nu
\geq
\gamma^{2}\int_{\Omega}d^{\alpha+2\gamma-2}((\widetilde{\nabla}d)d)
\,q^{2}d\nu\\+
\frac{\gamma}{\alpha+2\gamma}\int_{\partial\Omega}q^{2}
\langle\widetilde{\nabla}d^{\alpha+2\gamma},
d\nu\rangle-\frac{\gamma}{\alpha+
2\gamma}\int_{\Omega}q^{2}\mathcal{L}d^{\alpha+2\gamma}d\nu,
\end{multline}
since $d>0$ and $(\widetilde{\nabla}q)q=|\nabla_{X}q|^{2}\geq 0.$
On the other hand, it can be readily checked that for a vector
field $X$ we have
\begin{multline*}
\frac{\gamma}{\alpha+2\gamma} X^2(d^{\alpha+2\gamma})=
\gamma X(d^{\alpha+2\gamma-1}Xd)=
\frac{\gamma}{2-\beta} X(d^{\alpha+2\gamma+\beta-2}X(d^{2-\beta})) \\
=\frac{\gamma}{2-\beta}(\alpha+2\gamma+\beta-2)d^{\alpha+2\gamma+\beta-3}(Xd) X(d^{2-\beta})+
\frac{\gamma}{2-\beta} d^{\alpha+2\gamma+\beta-2}X^2(d^{2-\beta}) \\
=\gamma (\alpha+2\gamma+\beta-2)d^{\alpha+2\gamma-2}(Xd)^2 +
\frac{\gamma}{2-\beta} d^{\alpha+2\gamma+\beta-2}X^2(d^{2-\beta}).
\end{multline*}
Consequently, we get the equality
\begin{equation}\label{H2a}
-\frac{\gamma}{\alpha+2\gamma}\mathcal{L}d^{\alpha+2\gamma}=
-\gamma(\alpha+2\gamma+\beta-2)d^{\alpha+2\gamma-2}(\widetilde{\nabla}d)d
-\frac{\gamma}{2-\beta}d^{\alpha+2\gamma+\beta-2}\mathcal{L}d^{2-\beta}.
\end{equation}
Since $q^{2}=d^{-2\gamma}u^{2},$ substituting \eqref{H2a} into \eqref{H1a} we obtain

$$\int_{\Omega}d^{\alpha}(\widetilde{\nabla}u)ud\nu\geq (-\gamma^{2}-\gamma(\alpha+\beta-2))\int_{\Omega}d^{\alpha-2}((\widetilde{\nabla}d)d) u^{2}d\nu$$
$$-\frac{\gamma}{2-\beta}\int_{\Omega}(\mathcal{L}d^{2-\beta})d^{\alpha+\beta-2}u^{2}dx
+
\frac{\gamma}{\alpha+2\gamma}\int_{\partial\Omega}d^{-2\gamma}u^{2}
\langle\widetilde{\nabla}d^{\alpha+2\gamma},
d\nu\rangle.$$
Taking $d=\Gamma^{\frac{1}{2-\beta}},\;\beta>2,$ concerning the second term we observe that
\begin{equation}\label{Lzero}
\int_{\Omega}(\mathcal{L}\Gamma)\Gamma^{\frac{\alpha+\beta-2}{2-\beta}}u^{2}dx=0,
\;\alpha>2-\beta,\;\beta>2,
\end{equation}
since $\Gamma=\Gamma_y$ is the fundamental solution to $\mathcal L$.
The above equality is clear when $y$ is outside of $\Omega$.
If $y$ belongs to $\Omega$ we have
$$\int_{\Omega}(\mathcal{L}\Gamma)\Gamma^{\frac{\alpha+\beta-2}{2-\beta}}u^{2}dx=
\Gamma^{\frac{\alpha+\beta-2}{2-\beta}}(y)u^{2}(y)=0,
\;\alpha>2-\beta,\;\beta>2,$$
since $\frac{\alpha+\beta-2}{2-\beta}<0$ and $\frac{1}{\Gamma}(y)=0$ by {\rm (${\rm A}_{y}^{+}$)}.
Thus, with $d=\Gamma^{\frac{1}{2-\beta}},\;\beta>2,$ we obtain
\begin{multline}
\int_{\Omega}\Gamma^{\frac{\alpha}{2-\beta}}
(\widetilde{\nabla}u)u\,d\nu\geq
(-\gamma^{2}-\gamma(\alpha+\beta-2)) \int_{\Omega}
\Gamma^{\frac{\alpha-2}{2-\beta}}
(\widetilde{\nabla}\Gamma^{\frac{1}{2-\beta}})\Gamma^{\frac{1}{2-\beta}}
u^{2}\,d\nu
\\
+\frac{\gamma}{\alpha+2\gamma}
\int_{\partial\Omega}\Gamma^{-\frac{2\gamma}{2-\beta}}u^{2}
\langle\widetilde{\nabla}\Gamma^{\frac{\alpha+2\gamma}{2-\beta}},
d\nu\rangle.
\end{multline}
Taking
$\gamma=\frac{2-\beta-\alpha}{2},$
we obtain \eqref{LHa}. Finally, we note that with this $\gamma$, we have
$d^{\alpha+2\gamma}=\Gamma$, so that the use of
Proposition \ref{greenformulae} is justified.

Proof of \eqref{LH2} is similar to the above proof of \eqref{LH2a}. Recalling that $$(\widetilde{\nabla}u)u=\sum_{k=1}^{N} (X_{k}u)X_{k}u=|\nabla_{X}u|^{2},$$ inequality \eqref{LH2}
reduces to
\begin{multline}\label{LH8a}
\qquad \qquad \int_{\Omega}\Gamma^{\frac{\alpha}{2-\beta}}
(\widetilde{\nabla}u)u\,d\nu\geq
\left(\frac{\beta+\alpha-2}{2}\right)^{2}\int_{\Omega}
\Gamma^{\frac{\alpha-2}{2-\beta}}
(\widetilde{\nabla}\Gamma^{\frac{1}{2-\beta}})\Gamma^{\frac{1}{2-\beta}}
u^{2}\,d\nu\\+\frac{1}{4}\int_{\Omega}\Gamma^{\frac{\alpha-2}{2-\beta}}
(\widetilde{\nabla}\Gamma^{\frac{1}{2-\beta}})\Gamma^{\frac{1}{2-\beta}}\,
\left({\ln}\frac{R}{\Gamma^{\frac{1}{2-\beta}}}\right)^{-2}u^{2}d\nu
\\
+\frac{1}{2(\beta-2)}\int_{\partial\Omega}
\Gamma^{\frac{\alpha}{2-\beta}-1}\left({\ln}\frac{R}{\Gamma^{\frac{1}{2-\beta}}}\right)^{-1}u^{2}
\langle\widetilde{\nabla}\Gamma,
d\nu\rangle
\\+
\frac{\beta+\alpha-2}{2(\beta-2)}\int_{\partial\Omega}
\Gamma^{\frac{\alpha}{2-\beta}-1}u^{2}
\langle\widetilde{\nabla}\Gamma,
d\nu\rangle,
\end{multline}
which we will now prove.
Let us recall the first part of \eqref{H1a} as
\begin{multline}\label{H1}
\int_{\Omega}d^{\alpha}(\widetilde{\nabla}u)ud\nu=
\gamma^{2}
\int_{\Omega}d^{\alpha+2\gamma-2}((\widetilde{\nabla}d)d)\,q^{2}d\nu
+ \frac{\gamma}{\alpha+2\gamma}\int_{\Omega}(\widetilde{\nabla}d^{\alpha+2\gamma})q^{2}
d\nu\\+\int_{\Omega}d^{\alpha+2\gamma}(\widetilde{\nabla}q)qd\nu
 =
\gamma^{2}
\int_{\Omega}d^{\alpha+2\gamma-2}((\widetilde{\nabla}d)d)\,q^{2}d\nu+
\frac{\gamma}{\alpha+2\gamma}\int_{\partial\Omega}q^{2}
\langle\widetilde{\nabla}d^{\alpha+2\gamma},
d\nu\rangle
\\-\frac{\gamma}{\alpha+2\gamma}\int_{\Omega}q^{2}\mathcal{L}d^{\alpha+2\gamma}
d\nu+\int_{\Omega}d^{\alpha+2\gamma}(\widetilde{\nabla}q)qd\nu.
\end{multline}
Since $q^{2}=d^{-2\gamma}u^{2},$ substituting \eqref{H2a} into \eqref{H1} we obtain
\begin{multline*}\int_{\Omega}d^{\alpha}(\widetilde{\nabla}u)ud\nu= (-\gamma^{2}-\gamma(\alpha+\beta-2))\int_{\Omega}d^{\alpha-2}((\widetilde{\nabla}d)d) u^{2}d\nu\\-\frac{\gamma}{2-\beta}\int_{\Omega}(\mathcal{L}d^{2-\beta})d^{\alpha+\beta-2}u^{2}dx
+
\frac{\gamma}{\alpha+2\gamma}\int_{\partial\Omega}d^{-2\gamma}u^{2}
\langle\widetilde{\nabla}d^{\alpha+2\gamma},
d\nu\rangle+\int_{\Omega}d^{\alpha+2\gamma}(\widetilde{\nabla}q)qd\nu.
\end{multline*}
Using \eqref{Lzero}, with $d=\Gamma^{\frac{1}{2-\beta}},\;\beta>2,$ we obtain
\begin{multline}
\int_{\Omega}\Gamma^{\frac{\alpha}{2-\beta}}
(\widetilde{\nabla}u)u\,d\nu=
(-\gamma^{2}-\gamma(\alpha+\beta-2)) \int_{\Omega}
\Gamma^{\frac{\alpha-2}{2-\beta}}
(\widetilde{\nabla}\Gamma^{\frac{1}{2-\beta}})\Gamma^{\frac{1}{2-\beta}}
u^{2}\,d\nu
\\
+\frac{\gamma}{\alpha+2\gamma}
\int_{\partial\Omega}\Gamma^{-\frac{2\gamma}{2-\beta}}u^{2}
\langle\widetilde{\nabla}\Gamma^{\frac{\alpha+2\gamma}{2-\beta}},
d\nu\rangle+\int_{\Omega}d^{\alpha+2\gamma}(\widetilde{\nabla}q)qd\nu.
\end{multline}
Taking
$\gamma=\frac{2-\beta-\alpha}{2},$
we obtain
\begin{multline}\label{lterm}
\qquad \qquad \int_{\Omega}\Gamma^{\frac{\alpha}{2-\beta}}
(\widetilde{\nabla}u)u\,d\nu=
\left(\frac{\beta+\alpha-2}{2}\right)^{2}\int_{\Omega}
\Gamma^{\frac{\alpha-2}{2-\beta}}
(\widetilde{\nabla}\Gamma^{\frac{1}{2-\beta}})\Gamma^{\frac{1}{2-\beta}}
u^{2}\,d\nu\\+
\frac{\beta+\alpha-2}{2(\beta-2)}\int_{\partial\Omega}
\Gamma^{\frac{\alpha}{2-\beta}-1}u^{2}
\langle\widetilde{\nabla}\Gamma,
d\nu\rangle+\int_{\Omega}\Gamma(\widetilde{\nabla}q)qd\nu.
\end{multline}
Let us take $q=\left({\ln}\frac{R}{\Gamma^{\frac{1}{2-\beta}}}\right)^{\frac{1}{2}}\varphi$, that is,
$\varphi=\left({\ln}\frac{R}{\Gamma^{\frac{1}{2-\beta}}}\right)^{-\frac{1}{2}}\Gamma^{-\frac{2-\beta-\alpha}{2(2-\beta)}}u.$
A straightforward computation shows that
\begin{multline}\label{llterm}\int_{\Omega}\Gamma(\widetilde{\nabla}q)qd\nu=
\sum_{j=1}^{N}\int_{\Omega}\Gamma\left(X_{j}\left({\ln}\frac{R}{\Gamma^{\frac{1}{2-\beta}}}\right)^{\frac{1}{2}}\,\varphi+ \left({\ln}\frac{R}{\Gamma^{\frac{1}{2-\beta}}}\right)^{\frac{1}{2}}\,X_{j}\varphi\right)^{2}d\nu
\\
=
\frac{1}{4}\int_{\Omega}\Gamma^{\frac{-\beta}{2-\beta}}(\widetilde{\nabla}\Gamma^{\frac{1}{2-\beta}})
\Gamma^{\frac{1}{2-\beta}}\,
\left({\ln}\frac{R}{\Gamma^{\frac{1}{2-\beta}}}\right)^{-1}\varphi^{2}d\nu-
\int_{\Omega}\Gamma^{1-\frac{1}{2-\beta}}\varphi(\widetilde{\nabla}\Gamma^{\frac{1}{2-\beta}})\varphi d\nu\\
+\int_{\Omega}\Gamma\,
{\ln}\frac{R}{\Gamma^{\frac{1}{2-\beta}}}(\widetilde{\nabla}\varphi)\varphi d\nu
\\
=
\frac{1}{4}\int_{\Omega}\Gamma^{\frac{-\beta}{2-\beta}}(\widetilde{\nabla}\Gamma^{\frac{1}{2-\beta}})\Gamma^{\frac{1}{2-\beta}}\,
\left({\ln}\frac{R}{\Gamma^{\frac{1}{2-\beta}}}\right)^{-1}\varphi^{2}d\nu+\frac{1}{2(\beta-2)}
\int_{\Omega}(\widetilde{\nabla}\Gamma)\varphi^{2}d\nu\\
+\int_{\Omega}\Gamma\,
{\ln}\frac{R}{\Gamma^{\frac{1}{2-\beta}}}(\widetilde{\nabla}\varphi)\varphi d\nu
=
\frac{1}{4}\int_{\Omega}\Gamma^{\frac{-\beta}{2-\beta}}(\widetilde{\nabla}\Gamma^{\frac{1}{2-\beta}})\Gamma^{\frac{1}{2-\beta}}\,
\left({\ln}\frac{R}{\Gamma^{\frac{1}{2-\beta}}}\right)^{-1}\varphi^{2}d\nu\\
+\frac{1}{2(\beta-2)}\int_{\Omega}\mathcal{L}\Gamma\,\varphi^{2}d\nu+\frac{1}{2(\beta-2)}\int_{\partial\Omega}
\varphi^{2}
\langle\widetilde{\nabla}\Gamma,
d\nu\rangle+
\int_{\Omega}\Gamma\,
{\ln}\frac{R}{\Gamma^{\frac{1}{2-\beta}}}(\widetilde{\nabla}\varphi)\varphi d\nu.
\end{multline}
Since the second integral term of the right hand side vanishes  and the last integral term is positive from \eqref{llterm} we obtain that
\begin{multline}\label{lllterm}\int_{\Omega}\Gamma(\widetilde{\nabla}q)qd\nu\geq
\frac{1}{4}\int_{\Omega}\Gamma^{\frac{-\beta}{2-\beta}}(\widetilde{\nabla}\Gamma^{\frac{1}{2-\beta}})\Gamma^{\frac{1}{2-\beta}}\,
\left({\ln}\frac{R}{\Gamma^{\frac{1}{2-\beta}}}\right)^{-1}\varphi^{2}d\nu+\frac{1}{2(\beta-2)}\int_{\partial\Omega}
\varphi^{2}
\langle\widetilde{\nabla}\Gamma,
d\nu\rangle\\
=\frac{1}{4}\int_{\Omega}\Gamma^{\frac{\alpha-2}{2-\beta}}
(\widetilde{\nabla}\Gamma^{\frac{1}{2-\beta}})\Gamma^{\frac{1}{2-\beta}}\,
\left({\ln}\frac{R}{\Gamma^{\frac{1}{2-\beta}}}\right)^{-2}u^{2}d\nu
\\+\frac{1}{2(\beta-2)}\int_{\partial\Omega}
\Gamma^{\frac{\alpha}{2-\beta}-1}\left({\ln}\frac{R}{\Gamma^{\frac{1}{2-\beta}}}\right)^{-1}u^{2}
\langle\widetilde{\nabla}\Gamma,
d\nu\rangle
\end{multline}
Finally, \eqref{lterm} and \eqref{lllterm} imply \eqref{LH8a}.
\end{proof}

Even if $y\in\partial\Omega$, the statements of Theorem \ref{LHardy} remain true if
$y\not\in\partial\Omega\cap{\rm supp }\, u$.
Theorem \ref{LHardy} implies the following local
uncertainly principles:
\begin{cor}[Uncertainly principle on $\Omega$]\label{Luncertaintya}
Let $y\in M$ be such that {\rm (${\rm A}_{y}^{+}$)} holds with the fundamental solution $\Gamma=\Gamma_{y}$ in $T_{y}$.
Let $\Omega\subset T_{y}$ be an admissible domain and let $u\in C^{1}(\Omega)\bigcap C(\overline{\Omega})$.
Then for $\beta>2$ we have
\begin{multline}\label{UP1a}
\int_{\Omega}\Gamma^{\frac{2}{2-\beta}}|\nabla_{X}\Gamma^{\frac{1}{2-\beta}}|
^{2} |u|^{2}d\nu \int_{\Omega}|\nabla_{X}u|^{2}d\nu
\geq\left(\frac{\beta-2}{2}\right)^{2}\left(\int_{\Omega}
|\nabla_{X}\Gamma^{\frac{1}{2-\beta}}|^{2} |u|^{2}
d\nu\right)^{2}\\+ \frac{1}{2}
\int_{\partial\Omega}\Gamma^{-1}|u|^{2}
\langle\widetilde{\nabla}\Gamma,
d\nu\rangle\int_{\Omega}\Gamma^{\frac{2}{2-\beta}}
|\nabla_{X}\Gamma^{\frac{1}{2-\beta}}|^{2} |u|^{2}d\nu,
\end{multline}
and also
\begin{multline}\label{UP2a}
\int_{\Omega}\frac{\Gamma^{\frac{2}{2-\beta}}}{|\nabla_{X}\Gamma^{\frac{1}{2-\beta}}|
^{2}}|u|^{2}d\nu\int_{\Omega}|\nabla_{X}u|^{2}d\nu
\geq\left(\frac{\beta-2}{2}\right)^{2}\left(\int_{\Omega}|u|^{2}d\nu\right)^{2}
\\+
\frac{1}{2} \int_{\partial\Omega}\Gamma^{-1}|u|^{2}
\langle\widetilde{\nabla}\Gamma,
d\nu\rangle\,\int_{\Omega}\frac{\Gamma^{\frac{2}{2-\beta}}}
{|\nabla_{X}\Gamma^{\frac{1}{2-\beta}}|^{2}}|u|^{2}d\nu.
\end{multline}
\end{cor}
 As in \eqref{EQ:btn}, the last (boundary) terms in \eqref{UP1a} and \eqref{UP2a} can also be positive,
 thus providing generalisations but also refinements for uncertainty principles with respect to the boundary conditions.

\begin{proof}[Proof of Corollary \ref{Luncertaintya}]
Taking $\alpha=0$
in the inequality \eqref{LH2a} we get
$$ \int_{\Omega}\Gamma^{\frac{2}{2-\beta}}
|\nabla_{X}\Gamma^{\frac{1}{2-\beta}}|^{2}
|u|^{2}d\nu \int_{\Omega}|\nabla_{X}u|^{2}d\nu$$
$$ \geq\left(\frac{\beta-2}{2}\right)^{2}\int_{\Omega}\Gamma^{\frac{2}{2-\beta}}
|\nabla_{X}\Gamma^{\frac{1}{2-\beta}}|^{2}|u|^{2}d\nu\int_{\Omega}
 \frac{|\nabla_{X}\Gamma^{\frac{1}{2-\beta}}|^{2}}{\Gamma^{\frac{2}{2-\beta}}}
|u|^{2}\,d\nu$$$$+ \frac{1}{2}
\int_{\partial\Omega}\Gamma^{-1}|u|^{2}
\langle\widetilde{\nabla}\Gamma,
d\nu\rangle\int_{\Omega}\Gamma^{\frac{2}{2-\beta}}
|\nabla_{X}\Gamma^{\frac{1}{2-\beta}}|^{2} |u|^{2}d\nu$$
$$
\geq\left(\frac{\beta-2}{2}\right)^{2}\left(\int_{\Omega}
|\nabla_{X}\Gamma^{\frac{1}{2-\beta}}|^{2}
|u|^{2}d\nu\right)^{2}$$$$+ \frac{1}{2}
\int_{\partial\Omega}\Gamma^{-1}|u|^{2}
\langle\widetilde{\nabla}\Gamma,
d\nu\rangle\int_{\Omega}\Gamma^{\frac{2}{2-\beta}}
|\nabla_{X}\Gamma^{\frac{1}{2-\beta}}|^{2} |u|^{2}d\nu,
$$
where we have used the H\"older inequality in the last line.
This shows \eqref{UP1a}. The proof of \eqref{UP2a} is similar.
\end{proof}

As in the example \eqref{HRn}, in the Euclidean case $M=\mathbb R^{n}$ with $\beta=n\geq 3$, we have $\Gamma^{\frac{1}{2-\beta}}(x)=C|x|$ is the Euclidean distance,
so that $|\nabla \Gamma^{\frac{1}{2-\beta}}|=C$, and hence both
\eqref{UP1a} and \eqref{UP2a} reduce to the classical uncertainty principle
for $\Omega\subset\mathbb R^{n}$ if $u=0$ on $\partial\Omega$
(usually one takes $u\in C_{0}^{\infty}(\Omega)$):
\begin{equation*}\label{UPRna}
\int_{\Omega} |x|^{2} |u(x)|^{2}dx \int_{\Omega}|\nabla
u(x)|^{2}dx \geq\left(\frac{n-2}{2}\right)^{2}\left(\int_{\Omega}
 |u(x)|^{2} dx\right)^{2},\quad n\geq 3.
\end{equation*}
Similarly to the example of stratified Lie groups in \eqref{HRC}, e.g. now with boundary terms,
with $\beta=Q\geq 3$ the homogeneous dimension of the group, and
$\Gamma^{\frac{1}{2-\beta}}(x)=d(x)$ a quasi-distance on the group, for example
\eqref{UP1a} reduces to
\begin{multline*}
\int_{\Omega}d^{2}|\nabla_X d|^{2} |u|^{2}d\nu
\int_{\Omega}|\nabla_X u|^{2}d\nu
\geq\left(\frac{Q-2}{2}\right)^{2}\left(\int_{\Omega} |\nabla_X
d|^{2} |u|^{2} d\nu\right)^{2}\\+ \frac{1}{2}
\int_{\partial\Omega}d^{Q-2}|u|^{2}
\langle\widetilde{\nabla}d^{2-Q},
d\nu\rangle\int_{\Omega}d^{2}|\nabla_X d|^{2} |u|^{2}d\nu.
\end{multline*}
Again, if $u\in C_{0}^{\infty}(\Omega)$, the last (boundary) term disappears,
and one obtains the known uncertainty principle on the stratified Lie groups.

\begin{cor}[Uncertainly principle on $\Omega$ with double boundary terms]\label{Luncertainty}
Let $y\in M$ be such that {\rm (${\rm A}_{y}^{+}$)} holds with the fundamental solution $\Gamma=\Gamma_{y}$ in $T_{y}$.
Let $\Omega\subset T_{y},\,y\not\in\partial\Omega,$ be an admissible domain and let $u\in C^{1}(\Omega)\bigcap C(\overline{\Omega})$.
Then for $\beta>2$ we have
\begin{multline}\label{UP1}
\int_{\Omega}\Gamma^{\frac{2}{2-\beta}}|\nabla_{X}\Gamma^{\frac{1}{2-\beta}}|
^{2} |u|^{2}d\nu \int_{\Omega}|\nabla_{X}u|^{2}d\nu
\geq\left(\frac{\beta-2}{2}\right)^{2}\left(\int_{\Omega}
|\nabla_{X}\Gamma^{\frac{1}{2-\beta}}|^{2} |u|^{2}
d\nu\right)^{2}\\+ \frac{1}{4}\int_{\Omega}\frac{|\nabla_{X} \Gamma^{\frac{1}{2-\beta}}|^{2}}{\Gamma^{\frac{2}{2-\beta}}}
\,
\left({\ln}\frac{R}{\Gamma^{\frac{1}{2-\beta}}}\right)^{-2}|u|^{2}d\nu\int_{\Omega}\Gamma^{\frac{2}{2-\beta}}
|\nabla_{X}\Gamma^{\frac{1}{2-\beta}}|^{2} |u|^{2}d\nu
\\+\frac{1}{2(\beta-2)}\int_{\partial\Omega}
\Gamma^{-1}\left({\ln}\frac{R}{\Gamma^{\frac{1}{2-\beta}}}\right)^{-1}|u|^{2}
\langle\widetilde{\nabla}\Gamma,
d\nu\rangle\int_{\Omega}\Gamma^{\frac{2}{2-\beta}}
|\nabla_{X}\Gamma^{\frac{1}{2-\beta}}|^{2} |u|^{2}d\nu
\\+\frac{1}{2}
\int_{\partial\Omega}\Gamma^{-1}|u|^{2}
\langle\widetilde{\nabla}\Gamma,
d\nu\rangle\int_{\Omega}\Gamma^{\frac{2}{2-\beta}}
|\nabla_{X}\Gamma^{\frac{1}{2-\beta}}|^{2} |u|^{2}d\nu,
\end{multline}
and also
\begin{multline}\label{UP2}
\int_{\Omega}\frac{\Gamma^{\frac{2}{2-\beta}}}{|\nabla_{X}\Gamma^{\frac{1}{2-\beta}}|
^{2}}|u|^{2}d\nu\int_{\Omega}|\nabla_{X}u|^{2}d\nu
\geq\left(\frac{\beta-2}{2}\right)^{2}\left(\int_{\Omega}|u|^{2}d\nu\right)^{2}
\\+
\frac{1}{4}\int_{\Omega}\frac{|\nabla_{X} \Gamma^{\frac{1}{2-\beta}}|^{2}}{\Gamma^{\frac{2}{2-\beta}}}\,
\left({\ln}\frac{R}{\Gamma^{\frac{1}{2-\beta}}}\right)^{-2}|u|^{2}d\nu\int_{\Omega}\frac{\Gamma^{\frac{2}{2-\beta}}}
{|\nabla_{X}\Gamma^{\frac{1}{2-\beta}}|^{2}}|u|^{2}d\nu
\\+\frac{1}{2(\beta-2)}\int_{\partial\Omega}
\Gamma^{-1}\left({\ln}\frac{R}{\Gamma^{\frac{1}{2-\beta}}}\right)^{-1}|u|^{2}
\langle\widetilde{\nabla}\Gamma,
d\nu\rangle\int_{\Omega}\frac{\Gamma^{\frac{2}{2-\beta}}}
{|\nabla_{X}\Gamma^{\frac{1}{2-\beta}}|^{2}}|u|^{2}d\nu
\\+\frac{1}{2} \int_{\partial\Omega}\Gamma^{-1}|u|^{2}
\langle\widetilde{\nabla}\Gamma,
d\nu\rangle\,\int_{\Omega}\frac{\Gamma^{\frac{2}{2-\beta}}}
{|\nabla_{X}\Gamma^{\frac{1}{2-\beta}}|^{2}}|u|^{2}d\nu.
\end{multline}
\end{cor}

\begin{proof}[Proof of Corollary \ref{Luncertainty}]
Taking $\alpha=0$
in the inequality \eqref{LH2} we get
$$ \int_{\Omega}\Gamma^{\frac{2}{2-\beta}}
|\nabla_{X}\Gamma^{\frac{1}{2-\beta}}|^{2}
|u|^{2}d\nu \int_{\Omega}|\nabla_{X}u|^{2}d\nu$$
$$ \geq\left(\frac{\beta-2}{2}\right)^{2}\int_{\Omega}\Gamma^{\frac{2}{2-\beta}}
|\nabla_{X}\Gamma^{\frac{1}{2-\beta}}|^{2}|u|^{2}d\nu\int_{\Omega}
\frac{|\nabla_{X}\Gamma^{\frac{1}{2-\beta}}|^{2}}
{\Gamma^{\frac{2}{2-\beta}}}|u|^{2}\,d\nu$$

$$+ \frac{1}{4}\int_{\Omega}\frac{|\nabla_{X} \Gamma^{\frac{1}{2-\beta}}|^{2}}{\Gamma^{\frac{2}{2-\beta}}}\,
\left({\ln}\frac{R}{\Gamma^{\frac{1}{2-\beta}}}\right)^{-2}|u|^{2}d\nu\int_{\Omega}\Gamma^{\frac{2}{2-\beta}}
|\nabla_{X}\Gamma^{\frac{1}{2-\beta}}|^{2} |u|^{2}d\nu
$$
$$+\frac{1}{2(\beta-2)}\int_{\partial\Omega}
\Gamma^{-1}\left({\ln}\frac{R}{\Gamma^{\frac{1}{2-\beta}}}\right)^{-1}|u|^{2}
\langle\widetilde{\nabla}\Gamma,
d\nu\rangle\int_{\Omega}\Gamma^{\frac{2}{2-\beta}}
|\nabla_{X}\Gamma^{\frac{1}{2-\beta}}|^{2} |u|^{2}d\nu
$$
$$+\frac{1}{2}
\int_{\partial\Omega}\Gamma^{-1}|u|^{2}
\langle\widetilde{\nabla}\Gamma,
d\nu\rangle\int_{\Omega}\Gamma^{\frac{2}{2-\beta}}
|\nabla_{X}\Gamma^{\frac{1}{2-\beta}}|^{2} |u|^{2}d\nu$$
$$
\geq\left(\frac{\beta-2}{2}\right)^{2}\left(\int_{\Omega}
|\nabla_{X}\Gamma^{\frac{1}{2-\beta}}|^{2}
|u|^{2}d\nu\right)^{2}$$$$+ \frac{1}{4}\int_{\Omega}\frac{|\nabla_{X} \Gamma^{\frac{1}{2-\beta}}|^{2}}{\Gamma^{\frac{2}{2-\beta}}}\,
\left({\ln}\frac{R}{\Gamma^{\frac{1}{2-\beta}}}\right)^{-2}|u|^{2}d\nu\int_{\Omega}\Gamma^{\frac{2}{2-\beta}}
|\nabla_{X}\Gamma^{\frac{1}{2-\beta}}|^{2} |u|^{2}d\nu
$$$$+\frac{1}{2(\beta-2)}\int_{\partial\Omega}
\Gamma^{-1}\left({\ln}\frac{R}{\Gamma^{\frac{1}{2-\beta}}}\right)^{-1}|u|^{2}
\langle\widetilde{\nabla}\Gamma,
d\nu\rangle\int_{\Omega}\Gamma^{\frac{2}{2-\beta}}
|\nabla_{X}\Gamma^{\frac{1}{2-\beta}}|^{2} |u|^{2}d\nu
$$$$+\frac{1}{2}
\int_{\partial\Omega}\Gamma^{-1}|u|^{2}
\langle\widetilde{\nabla}\Gamma,
d\nu\rangle\int_{\Omega}\Gamma^{\frac{2}{2-\beta}}
|\nabla_{X}\Gamma^{\frac{1}{2-\beta}}|^{2} |u|^{2}d\nu,
$$
where we have used the H\"older inequality.
This shows \eqref{UP1}. The proof of \eqref{UP2} is similar.
\end{proof}

In the Euclidean case $M=\mathbb R^{n}$ with $\beta=n\geq 3$, we have $\Gamma^{\frac{1}{2-\beta}}(x)=C|x|$ is the Euclidean distance,
so that $|\nabla \Gamma^{\frac{1}{2-\beta}}|=C$, and hence both
\eqref{UP1} and \eqref{UP2} reduce to the improved uncertainty principle
for $\Omega\subset\mathbb R^{n}$ if $u=0$ on $\partial\Omega$
(usually one takes $u\in C_{0}^{\infty}(\Omega)$):
\begin{multline*}
\int_{\Omega} |x|^{2} |u(x)|^{2}dx \int_{\Omega}|\nabla
u(x)|^{2}dx \geq\left(\frac{n-2}{2}\right)^{2}\left(\int_{\Omega}
 |u(x)|^{2} dx\right)^{2}\\
 +\frac{1}{4}\int_{\Omega}\frac{1}{|x|^{2}}\,\left({\ln}\frac{R}{|x|}\right)^{-2}|u(x)|^{2}d\nu\int_{\Omega}|x|^{2}
|u(x)|^{2}d\nu,\quad n\geq 3.
\end{multline*}
Similarly to the example of stratified Lie groups in \eqref{HRC}, e.g. now with boundary terms,
with $\beta=Q\geq 3$ the homogeneous dimension of the group $\mathbb{G}$, and
$\Gamma^{\frac{1}{2-\beta}}(x)=d(x)$ a quasi-distance on the group, for example
\eqref{UP1} reduces to
\begin{multline*}
\int_{\Omega}d^{2}|\nabla_X d|^{2} |u|^{2}d\nu
\int_{\Omega}|\nabla_X u|^{2}d\nu
\geq\left(\frac{Q-2}{2}\right)^{2}\left(\int_{\Omega} |\nabla_X
d|^{2} |u|^{2} d\nu\right)^{2}\\+ \frac{1}{4}\int_{\Omega}\frac{|\nabla_{X} d|^{2}}{d^{2}}
\,
\left({\ln}\frac{R}{d}\right)^{-2}|u|^{2}d\nu\int_{\Omega}d^{2}
|\nabla_{X}d|^{2} |u|^{2}d\nu
\\
+\frac{1}{2(Q-2)}\int_{\partial\Omega}
d^{Q-2}\left({\ln}\frac{R}{d}\right)^{-1}|u|^{2}
\langle\widetilde{\nabla}d^{2-Q},
d\nu\rangle\int_{\Omega}d^{2}
|\nabla_{X}d|^{2} |u|^{2}d\nu
\\
\\+\frac{1}{2}
\int_{\partial\Omega}d^{Q-2}|u|^{2}
\langle\widetilde{\nabla}d^{2-Q},
d\nu\rangle\int_{\Omega}d^{2}|\nabla_X d|^{2} |u|^{2}d\nu.
\end{multline*}
Again, if $u\in C_{0}^{\infty}(\mathbb{G})$, the last terms disappear,
and one obtains the improved uncertainty principle on stratified Lie groups.

\section{Local Rellich inequalities}
\label{SecR}

We now present local refined versions of Rellich
inequalities with additional boundary terms
on the right hand side.

\begin{thm}\label{LRellich} Let $y\in M$ be such that {\rm (${\rm A}_{y}^{+}$)} holds with the fundamental solution $\Gamma=\Gamma_{y}$ in $T_{y}$.
Let $\Omega\subset T_{y},\,y\not\in\partial\Omega,$ be a strongly admissible domain,
$\alpha\in \mathbb{R}$,\ $\beta>\alpha>4-\beta$, $\beta>2$ and $R\geq\, e\, {\rm sup}_{\Omega}\Gamma^{\frac{1}{2-\beta}}$. Let
$u\in C^{2}(\Omega)\bigcap C^{1}(\overline{\Omega})$.
 Then the following generalised local Rellich
inequalities are valid:
\begin{multline}\label{LR2a}
\qquad \qquad \int_{\Omega}\frac{\Gamma^{\frac{\alpha}{2-\beta}}}
{|\nabla_{X}\Gamma^{\frac{1}{2-\beta}}|^{2}}
|\mathcal{L}u|^{2}d\nu\geq \frac{(\beta+\alpha-4)^{2}(\beta-\alpha)^{2}}{16}\int_{\Omega}
\Gamma^{\frac{\alpha-4}{2-\beta}} |\nabla_{X} \Gamma^{\frac{1}{2-\beta}}|^{2}
|u|^{2}\,d\nu
\\+
\frac{(\beta+\alpha-4)^{2}(\beta-\alpha)}{4(\beta-2)}\int_{\partial\Omega}
\Gamma^{\frac{\alpha-2}{2-\beta}-1}|u|^{2}
\langle\widetilde{\nabla}\Gamma,
d\nu\rangle+\frac{(\beta+\alpha-4)(\beta-\alpha)}{4}\mathcal{C}(u),
\end{multline}
and its further refinement
\begin{multline}\label{LR2}
\qquad \qquad \int_{\Omega}\frac{\Gamma^{\frac{\alpha}{2-\beta}}}
{|\nabla_{X}\Gamma^{\frac{1}{2-\beta}}|^{2}}
|\mathcal{L}u|^{2}d\nu\geq \frac{(\beta+\alpha-4)^{2}(\beta-\alpha)^{2}}{16}\int_{\Omega}
\Gamma^{\frac{\alpha-4}{2-\beta}} |\nabla_{X} \Gamma^{\frac{1}{2-\beta}}|^{2}
|u|^{2}\,d\nu\\+\frac{(\beta+\alpha-4)(\beta-\alpha)}{8}\int_{\Omega}\Gamma^{\frac{\alpha-4}{2-\beta}}
|\nabla_{X} \Gamma^{\frac{1}{2-\beta}}|^{2}\,
\left({\ln}\frac{R}{\Gamma^{\frac{1}{2-\beta}}}\right)^{-2}|u|^{2}d\nu
\\+\frac{(\beta+\alpha-4)(\beta-\alpha)}{4(\beta-2)}\int_{\partial\Omega}
\Gamma^{\frac{\alpha-2}{2-\beta}-1}\left({\ln}\frac{R}{\Gamma^{\frac{1}{2-\beta}}}\right)^{-1}|u|^{2}
\langle\widetilde{\nabla}\Gamma,
d\nu\rangle
\\+
\frac{(\beta+\alpha-4)^{2}(\beta-\alpha)}{4(\beta-2)}\int_{\partial\Omega}
\Gamma^{\frac{\alpha-2}{2-\beta}-1}|u|^{2}
\langle\widetilde{\nabla}\Gamma,
d\nu\rangle+\frac{(\beta+\alpha-4)(\beta-\alpha)}{4}\mathcal{C}(u),
\end{multline}
where $\nabla_{X}=(X_{1},\ldots,X_{N})$ and
$$\mathcal{C}(u):=\frac{\alpha-2}{2-\beta}\int_{\partial\Omega}
u^{2}\Gamma^{\frac{\alpha-2}{2-\beta}-1}\langle\widetilde{\nabla}\Gamma,
d\nu\rangle-2\int_{\partial\Omega}
\Gamma^{\frac{\alpha-2}{2-\beta}}u\langle\widetilde{\nabla}u,
d\nu\rangle.$$
\end{thm}

\begin{proof}[Proof of Theorem \ref{LRellich}]
Proof of \eqref{LR2a}.  A direct calculation shows that
\begin{multline*}
\mathcal{L}\Gamma^{\frac{\alpha-2}{2-\beta}}=\sum_{k=1}^{N}X_{k}^{2}\Gamma^{\frac{\alpha-2}{2-\beta}}
=(\alpha-2)\sum_{k=1}^{N}X_{k}\left(\Gamma^{\frac{\alpha-3}{2-\beta}}X_{k}\Gamma^{\frac{1}{2-\beta}}\right)
\\=(\alpha-2)(\alpha-3)\Gamma^{\frac{\alpha-4}{2-\beta}}
\sum_{k=1}^{N}\left|
X_{k}\Gamma^{\frac{1}{2-\beta}}\right|^{2}
+(\alpha-2)\Gamma^{\frac{\alpha-3}{2-\beta}}
\sum_{k=1}^{N}
X_{k}\left(X_{k}\Gamma^{\frac{1}{2-\beta}}\right)
\\=(\alpha-2)(\alpha-3)\Gamma^{\frac{\alpha-4}{2-\beta}}
\sum_{k=1}^{N}\left|
X_{k}\Gamma^{\frac{1}{2-\beta}}\right|^{2}
+\frac{\alpha-2}{2-\beta}\Gamma^{\frac{\alpha-3}{2-\beta}}
\sum_{k=1}^{N}
X_{k}\left(\Gamma^{\frac{\beta-1}{2-\beta}}X_{k}\Gamma\right)
\\=(\alpha-2)(\alpha-3)\Gamma^{\frac{\alpha-4}{2-\beta}}
\sum_{k=1}^{N}\left|
X_{k}\Gamma^{\frac{1}{2-\beta}}\right|^{2}
+\frac{(\alpha-2)(\beta-1)}{2-\beta}\Gamma^{\frac{\alpha-3}{2-\beta}}
\Gamma^{-1}\sum_{k=1}^{N}
(X_{k}\Gamma^{\frac{1}{2-\beta}})(X_{k}\Gamma)
\\+\frac{\alpha-2}{2-\beta}
\Gamma^{\frac{\beta+\alpha-4}{2-\beta}}\mathcal{L}\Gamma=
(\alpha-2)(\alpha-3)\Gamma^{\frac{\alpha-4}{2-\beta}}
\sum_{k=1}^{N}\left|
X_{k}\Gamma^{\frac{1}{2-\beta}}\right|^{2}
\\+(\alpha-2)(\beta-1)\Gamma^{\frac{\alpha-4}{2-\beta}}
\sum_{k=1}^{N}
(X_{k}\Gamma^{\frac{1}{2-\beta}})(X_{k}\Gamma^{\frac{1}{2-\beta}})
+\frac{\alpha-2}{2-\beta}
\Gamma^{\frac{\beta+\alpha-4}{2-\beta}}\mathcal{L}\Gamma
\\=(\beta+\alpha-4)(\alpha-2)\Gamma^{\frac{\alpha-4}{2-\beta}}
|\nabla_{X}\Gamma^{\frac{1}{2-\beta}}|^{2}+
\frac{\alpha-2}{2-\beta}\Gamma^{\frac{\beta+\alpha-4}{2-\beta}}\mathcal{L}\Gamma,
\end{multline*}
that is,
\begin{equation}\label{mainequility}
\mathcal{L}\Gamma^{\frac{\alpha-2}{2-\beta}}=(\beta+\alpha-4)(\alpha-2)\Gamma^{\frac{\alpha-4}{2-\beta}}
|\nabla_{X}\Gamma^{\frac{1}{2-\beta}}|^{2}+
\frac{\alpha-2}{2-\beta}\Gamma^{\frac{\beta+\alpha-4}{2-\beta}}\mathcal{L}\Gamma.
\end{equation}

As before we can assume that $u$ is real-valued. Multiplying both sides of \eqref{mainequility} by $u^{2}$ and integrating over $\Omega$, since $u$ is the fundamental solution of $\mathcal{L}$ and $\beta+\alpha-4>0$, we obtain
\begin{equation}\label{RF1}
\int_{\Omega}u^{2}\mathcal{L}\Gamma^{\frac{\alpha-2}{2-\beta}}\,d\nu=
(\beta+\alpha-4)(\alpha-2)\int_{\Omega}\Gamma^{\frac{\alpha-4}{2-\beta}}
|\nabla_{X}\Gamma^{\frac{1}{2-\beta}}|^{2}u^{2}\,d\nu.
\end{equation}
On the other hand, by using the Green's second formula \eqref{g2}, we have
\begin{multline}\label{RF2}
\int_{\Omega}u^{2}\mathcal{L}\Gamma^{\frac{\alpha-2}{2-\beta}}\,d\nu=
\int_{\Omega}\Gamma^{\frac{\alpha-2}{2-\beta}}\mathcal{L}u^{2}\,d\nu
+\int_{\partial\Omega}
u^{2}\langle\widetilde{\nabla}\Gamma^{\frac{\alpha-2}{2-\beta}},
d\nu\rangle-\int_{\partial\Omega}
\Gamma^{\frac{\alpha-2}{2-\beta}}\langle\widetilde{\nabla}u^{2},
d\nu\rangle\\
=
\int_{\Omega}\Gamma^{\frac{\alpha-2}{2-\beta}}(2u\mathcal{L}u+2|\nabla_{X} u|^{2})\,d\nu
+\mathcal{C}(u),
\end{multline}
where
$$\mathcal{C}(u):=\frac{\alpha-2}{2-\beta}\int_{\partial\Omega}
u^{2}\Gamma^{\frac{\alpha-2}{2-\beta}-1}\langle\widetilde{\nabla}\Gamma,
d\nu\rangle-\int_{\partial\Omega}
2\Gamma^{\frac{\alpha-2}{2-\beta}}u\langle\widetilde{\nabla}u,
d\nu\rangle.$$
Combining \eqref{RF1} and \eqref{RF2} we obtain
\begin{multline}\label{R2a}
-2\int_{\Omega}\Gamma^{\frac{\alpha-2}{2-\beta}}u \mathcal{L}ud\nu+(\beta+\alpha-4)(\alpha-2)          \int_{\Omega}\Gamma^{\frac{\alpha-4}{2-\beta}}|\nabla_{X} \Gamma^{\frac{1}{2-\beta}}|^{2}\,u^{2}d\nu\\=2          \int_{\Omega}\Gamma^{\frac{\alpha-2}{2-\beta}}|\nabla_{X} u|^{2}d\nu+\mathcal{C}(u).
\end{multline}
By using \eqref{LH2a} we obtain
\begin{multline}
-2\int_{\Omega}\Gamma^{\frac{\alpha-2}{2-\beta}}u \mathcal{L}ud\nu+(\beta+\alpha-4)(\alpha-2)          \int_{\Omega}\Gamma^{\frac{\alpha-4}{2-\beta}}|\nabla_{X} \Gamma^{\frac{1}{2-\beta}}|^{2}\,|u|^{2}d\nu
\\\geq
2\left(\frac{\beta+\alpha-4}{2}\right)^{2}\int_{\Omega}
\Gamma^{\frac{\alpha-4}{2-\beta}} |\nabla_{X} \Gamma^{\frac{1}{2-\beta}}|^{2}
|u|^{2}\,d\nu\\+
\frac{\beta+\alpha-4}{\beta-2}\int_{\partial\Omega}
\Gamma^{\frac{\alpha-2}{2-\beta}-1}|u|^{2}
\langle\widetilde{\nabla}\Gamma,
d\nu\rangle+\mathcal{C}(u).
\end{multline}
It follows that
\begin{multline}\label{R3a}
-\int_{\Omega}\Gamma^{\frac{\alpha-2}{2-\beta}}u \mathcal{L}ud\nu\geq
\left(\frac{\beta+\alpha-4}{2}\right)\left(\frac{\beta-\alpha}{2}\right)\int_{\Omega}
\Gamma^{\frac{\alpha-4}{2-\beta}} |\nabla_{X} \Gamma^{\frac{1}{2-\beta}}|^{2}
|u|^{2}\,d\nu\\+
\frac{\beta+\alpha-4}{2(\beta-2)}\int_{\partial\Omega}
\Gamma^{\frac{\alpha-2}{2-\beta}-1}|u|^{2}
\langle\widetilde{\nabla}\Gamma,
d\nu\rangle+\frac{1}{2}\mathcal{C}(u).
\end{multline}
On the other hand, for any $\epsilon>0$ H\"older's and Young's inequalities give
\begin{multline}\label{R4a} -\int_{\Omega}\Gamma^{\frac{\alpha-2}{2-\beta}}u \mathcal{L}ud\nu\leq
\left(\int_{\Omega}\Gamma^{\frac{\alpha-4}{2-\beta}}
|\nabla_{X}\Gamma^{\frac{1}{2-\beta}}|^{2}|u|^{2}d\nu\right)^{\frac{1}{2}}
\left(\int_{\Omega}\frac{\Gamma^{\frac{\alpha}{2-\beta}}}{|\nabla_{X}\Gamma^{\frac{1}{2-\beta}}|^{2}}
|\mathcal{L}u|^{2}d\nu\right)^{\frac{1}{2}}\\
\leq
\epsilon
\int_{\Omega}\Gamma^{\frac{\alpha-4}{2-\beta}}
|\nabla_{X}\Gamma^{\frac{1}{2-\beta}}|^{2}|u|^{2}d\nu+
\frac{1}{4\epsilon}\int_{\Omega}\frac{\Gamma^{\frac{\alpha}{2-\beta}}}
{|\nabla_{X}\Gamma^{\frac{1}{2-\beta}}|^{2}}
|\mathcal{L}u|^{2}d\nu.
\end{multline}
Inequalities \eqref{R4a} and \eqref{R3a} imply that
\begin{multline*}
\int_{\Omega}\frac{\Gamma^{\frac{\alpha}{2-\beta}}}
{|\nabla_{X}\Gamma^{\frac{1}{2-\beta}}|^{2}}
|\mathcal{L}u|^{2}d\nu\geq
\left(-4\epsilon^{2}+(\beta+\alpha-4)(\beta-\alpha)\epsilon\right)\int_{\Omega}
\Gamma^{\frac{\alpha-4}{2-\beta}} |\nabla_{X} \Gamma^{\frac{1}{2-\beta}}|^{2}
|u|^{2}\,d\nu\\+
\frac{2(\beta+\alpha-4)\epsilon}{\beta-2}\int_{\partial\Omega}
\Gamma^{\frac{\alpha-2}{2-\beta}-1}|u|^{2}
\langle\widetilde{\nabla}\Gamma,
d\nu\rangle+2\epsilon\mathcal{C}(u).
\end{multline*}
Taking
$\epsilon=\frac{(\beta+\alpha-4)(\beta-\alpha)}{8},$ we obtain
\begin{multline*}
\int_{\Omega}\frac{\Gamma^{\frac{\alpha}{2-\beta}}}
{|\nabla_{X}\Gamma^{\frac{1}{2-\beta}}|^{2}}
|\mathcal{L}u|^{2}d\nu\geq \frac{(\beta+\alpha-4)^{2}(\beta-\alpha)^{2}}{16}\int_{\Omega}
\Gamma^{\frac{\alpha-4}{2-\beta}} |\nabla_{X} \Gamma^{\frac{1}{2-\beta}}|^{2}
|u|^{2}\,d\nu
\\+
\frac{(\beta+\alpha-4)^{2}(\beta-\alpha)}{4(\beta-2)}\int_{\partial\Omega}
\Gamma^{\frac{\alpha-2}{2-\beta}-1}|u|^{2}
\langle\widetilde{\nabla}\Gamma,
d\nu\rangle+\frac{(\beta+\alpha-4)(\beta-\alpha)}{4}\mathcal{C}(u).
\end{multline*}

Proof of \eqref{LR2}.
From \eqref{R2a}, by using \eqref{LH2}, we obtain
\begin{multline}\label{R2}
-2\int_{\Omega}\Gamma^{\frac{\alpha-2}{2-\beta}}u \mathcal{L}ud\nu+(\beta+\alpha-4)(\alpha-2)          \int_{\Omega}\Gamma^{\frac{\alpha-4}{2-\beta}}|\nabla_{X} \Gamma^{\frac{1}{2-\beta}}|^{2}\,|u|^{2}d\nu
\\\geq
2\left(\frac{\beta+\alpha-4}{2}\right)^{2}\int_{\Omega}
\Gamma^{\frac{\alpha-4}{2-\beta}} |\nabla_{X} \Gamma^{\frac{1}{2-\beta}}|^{2}
|u|^{2}\,d\nu\\+\frac{1}{2}\int_{\Omega}\Gamma^{\frac{\alpha-4}{2-\beta}}
|\nabla_{X} \Gamma^{\frac{1}{2-\beta}}|^{2}\,
\left({\ln}\frac{R}{\Gamma^{\frac{1}{2-\beta}}}\right)^{-2}|u|^{2}d\nu+\frac{1}{(\beta-2)}\int_{\partial\Omega}
\Gamma^{\frac{\alpha-2}{2-\beta}-1}\left({\ln}\frac{R}{\Gamma^{\frac{1}{2-\beta}}}\right)^{-1}|u|^{2}
\langle\widetilde{\nabla}\Gamma,
d\nu\rangle\\+
\frac{\beta+\alpha-4}{\beta-2}\int_{\partial\Omega}
\Gamma^{\frac{\alpha-2}{2-\beta}-1}|u|^{2}
\langle\widetilde{\nabla}\Gamma,
d\nu\rangle+\mathcal{C}(u).
\end{multline}
It follows that
\begin{multline}\label{R3}
-\int_{\Omega}\Gamma^{\frac{\alpha-2}{2-\beta}}u \mathcal{L}ud\nu\geq
\left(\frac{\beta+\alpha-4}{2}\right)\left(\frac{\beta-\alpha}{2}\right)\int_{\Omega}
\Gamma^{\frac{\alpha-4}{2-\beta}} |\nabla_{X} \Gamma^{\frac{1}{2-\beta}}|^{2}
|u|^{2}\,d\nu\\+\frac{1}{4}\int_{\Omega}\Gamma^{\frac{\alpha-4}{2-\beta}}
|\nabla_{X} \Gamma^{\frac{1}{2-\beta}}|^{2}\,
\left({\ln}\frac{R}{\Gamma^{\frac{1}{2-\beta}}}\right)^{-2}|u|^{2}d\nu
+\frac{1}{2(\beta-2)}\int_{\partial\Omega}
\Gamma^{\frac{\alpha-2}{2-\beta}-1}\left({\ln}\frac{R}{\Gamma^{\frac{1}{2-\beta}}}\right)^{-1}|u|^{2}
\langle\widetilde{\nabla}\Gamma,
d\nu\rangle
\\+
\frac{\beta+\alpha-4}{2(\beta-2)}\int_{\partial\Omega}
\Gamma^{\frac{\alpha-2}{2-\beta}-1}|u|^{2}
\langle\widetilde{\nabla}\Gamma,
d\nu\rangle+\frac{1}{2}\mathcal{C}(u).
\end{multline}

Inequalities \eqref{R4a} and \eqref{R3} imply that
\begin{multline*}
\int_{\Omega}\frac{\Gamma^{\frac{\alpha}{2-\beta}}}
{|\nabla_{X}\Gamma^{\frac{1}{2-\beta}}|^{2}}
|\mathcal{L}u|^{2}d\nu\geq
\left(-4\epsilon^{2}+(\beta+\alpha-4)(\beta-\alpha)\epsilon\right)\int_{\Omega}
\Gamma^{\frac{\alpha-4}{2-\beta}} |\nabla_{X} \Gamma^{\frac{1}{2-\beta}}|^{2}
|u|^{2}\,d\nu\\+\epsilon\int_{\Omega}\Gamma^{\frac{\alpha-4}{2-\beta}}
|\nabla_{X}\Gamma^{\frac{1}{2-\beta}}|^{2}\,
\left({\ln}\frac{R}{\Gamma^{\frac{1}{2-\beta}}}\right)^{-2}|u|^{2}d\nu
+\frac{2\epsilon}{(\beta-2)}\int_{\partial\Omega}
\Gamma^{\frac{\alpha-2}{2-\beta}-1}\left({\ln}\frac{R}{\Gamma^{\frac{1}{2-\beta}}}\right)^{-1}|u|^{2}
\langle\widetilde{\nabla}\Gamma,
d\nu\rangle\\+
\frac{2(\beta+\alpha-4)\epsilon}{\beta-2}\int_{\partial\Omega}
\Gamma^{\frac{\alpha-2}{2-\beta}-1}|u|^{2}
\langle\widetilde{\nabla}\Gamma,
d\nu\rangle+2\epsilon\mathcal{C}(u).
\end{multline*}
Taking
$\epsilon=\frac{(\beta+\alpha-4)(\beta-\alpha)}{8},$ we obtain
\begin{multline*}
\int_{\Omega}\frac{\Gamma^{\frac{\alpha}{2-\beta}}}
{|\nabla_{X}\Gamma^{\frac{1}{2-\beta}}|^{2}}
|\mathcal{L}u|^{2}d\nu\geq \frac{(\beta+\alpha-4)^{2}(\beta-\alpha)^{2}}{16}\int_{\Omega}
\Gamma^{\frac{\alpha-4}{2-\beta}} |\nabla_{X} \Gamma^{\frac{1}{2-\beta}}|^{2}
|u|^{2}\,d\nu\\+\frac{(\beta+\alpha-4)(\beta-\alpha)}{8}\int_{\Omega}\Gamma^{\frac{\alpha-4}{2-\beta}}
|\nabla_{X} \Gamma^{\frac{1}{2-\beta}}|^{2}\,
\left({\ln}\frac{R}{\Gamma^{\frac{1}{2-\beta}}}\right)^{-2}|u|^{2}d\nu
\\+\frac{(\beta+\alpha-4)(\beta-\alpha)}{4(\beta-2)}\int_{\partial\Omega}
\Gamma^{\frac{\alpha-2}{2-\beta}-1}\left({\ln}\frac{R}{\Gamma^{\frac{1}{2-\beta}}}\right)^{-1}|u|^{2}
\langle\widetilde{\nabla}\Gamma,
d\nu\rangle
\\+
\frac{(\beta+\alpha-4)^{2}(\beta-\alpha)}{4(\beta-2)}\int_{\partial\Omega}
\Gamma^{\frac{\alpha-2}{2-\beta}-1}|u|^{2}
\langle\widetilde{\nabla}\Gamma,
d\nu\rangle+\frac{(\beta+\alpha-4)(\beta-\alpha)}{4}\mathcal{C}(u).
\end{multline*}
This completes the proof.
\end{proof}

A modification and refinement of the proof yields another variant of an improved
Rellich inequality with boundary terms. In Remark \ref{strLiegroups} we will give simplified versions of
all the estimates in the setting of stratified Lie groups to clarify the differences in appearing weights and constants.

\begin{thm}\label{2LRellich} Let $y\in M$ be such that {\rm (${\rm A}_{y}^{+}$)} holds with the fundamental solution $\Gamma=\Gamma_{y}$ in $T_{y}$.
Let $\Omega\subset T_{y},\,y\not\in\partial\Omega,$ be a strongly admissible domain,
$\alpha\in \mathbb{R}$,\ $\beta>\alpha>\frac{8-\beta}{3}$, $\beta>2$ and $R\geq\, e\, {\rm sup}_{\Omega}\Gamma^{\frac{1}{2-\beta}}$. Let
$u\in C^{2}(\Omega)\bigcap C^{1}(\overline{\Omega})$.
 Then the following generalised local Rellich
inequalities are valid:
\begin{multline}\label{2LR2a}
\qquad \qquad \int_{\Omega}\frac{\Gamma^{\frac{\alpha}{2-\beta}}}
{|\nabla_{X}\Gamma^{\frac{1}{2-\beta}}|^{2}}|\mathcal{L}u|^{2}d\nu\geq
\frac{(\beta-\alpha)^{2}}{4}\int_{\Omega}\Gamma^{\frac{\alpha-2}{2-\beta}}|\nabla_{X} u|^{2}d\nu
\\+\frac{(\beta+3\alpha-8)(\beta+\alpha-4)(\beta-\alpha)}{8(\beta-2)}\int_{\partial\Omega}
\Gamma^{\frac{\alpha-2}{2-\beta}-1}|u|^{2}
\langle\widetilde{\nabla}\Gamma,
d\nu\rangle\\+
\frac{(\beta+\alpha-4)(\beta-\alpha)}{4}\mathcal{C}(u),
\end{multline}
and its further refinement
\begin{multline}\label{2LR2}
\qquad \qquad \int_{\Omega}\frac{\Gamma^{\frac{\alpha}{2-\beta}}}
{|\nabla_{X}\Gamma^{\frac{1}{2-\beta}}|^{2}}
|\mathcal{L}u|^{2}d\nu\geq \frac{(\beta-\alpha)^{2}}{4}\int_{\Omega}\Gamma^{\frac{\alpha-2}{2-\beta}}|\nabla_{X} u|^{2}d\nu\\+\frac{(\beta+3\alpha-8)(\beta-\alpha)}{16}\int_{\Omega}\Gamma^{\frac{\alpha-4}{2-\beta}}
|\nabla_{X} \Gamma^{\frac{1}{2-\beta}}|^{2}\,
\left({\ln}\frac{R}{\Gamma^{\frac{1}{2-\beta}}}\right)^{-2}|u|^{2}d\nu
\\
+\frac{(\beta+3\alpha-8)(\beta-\alpha)}
{8(\beta-2)}\int_{\partial\Omega}
\Gamma^{\frac{\alpha-2}{2-\beta}-1}\left({\ln}\frac{R}{\Gamma^{\frac{1}{2-\beta}}}\right)^{-1}|u|^{2}
\langle\widetilde{\nabla}\Gamma,
d\nu\rangle
\\+
\frac{(\beta+3\alpha-8)(\beta+\alpha-4)(\beta-\alpha)}{8(\beta-2)}\int_{\partial\Omega}
\Gamma^{\frac{\alpha-2}{2-\beta}-1}|u|^{2}
\langle\widetilde{\nabla}\Gamma,
d\nu\rangle \\
+\frac{(\beta+\alpha-4)(\beta-\alpha)}{4}\mathcal{C}(u),
\end{multline}
where $\nabla_{X}=(X_{1},\ldots,X_{N})$ and
$$\mathcal{C}(u):=\frac{\alpha-2}{2-\beta}\int_{\partial\Omega}
u^{2}\Gamma^{\frac{\alpha-2}{2-\beta}-1}\langle\widetilde{\nabla}\Gamma,
d\nu\rangle-2\int_{\partial\Omega}
\Gamma^{\frac{\alpha-2}{2-\beta}}u\langle\widetilde{\nabla}u,
d\nu\rangle.$$
\end{thm}

\begin{proof}[Proof of Theorem \ref{2LRellich}]
Proof of \eqref{2LR2a}. Let us rewrite \eqref{R2a} in the form
\begin{multline}\label{2R2a}
\frac{1}{2}\mathcal{C}(u)+\int_{\Omega}\Gamma^{\frac{\alpha-2}{2-\beta}}|\nabla_{X} u|^{2}d\nu=-\int_{\Omega}\Gamma^{\frac{\alpha-2}{2-\beta}}u \mathcal{L}ud\nu\\+\frac{(\beta+\alpha-4)(\alpha-2)}{2}        \int_{\Omega}\Gamma^{\frac{\alpha-4}{2-\beta}}|\nabla_{X} \Gamma^{\frac{1}{2-\beta}}|^{2}\,|u|^{2}d\nu        .
\end{multline}
Also recalling \eqref{R4a} we have
\begin{multline}\label{2R4a} -\int_{\Omega}\Gamma^{\frac{\alpha-2}{2-\beta}}u \mathcal{L}ud\nu
\leq
\epsilon
\int_{\Omega}\Gamma^{\frac{\alpha-4}{2-\beta}}
|\nabla_{X}\Gamma^{\frac{1}{2-\beta}}|^{2}|u|^{2}d\nu+
\frac{1}{4\epsilon}\int_{\Omega}\frac{\Gamma^{\frac{\alpha}{2-\beta}}}
{|\nabla_{X}\Gamma^{\frac{1}{2-\beta}}|^{2}}
|\mathcal{L}u|^{2}d\nu.
\end{multline}
Inequalities \eqref{2R4a} and \eqref{2R2a} imply that
\begin{multline}\label{2R5a}
\frac{1}{2}\mathcal{C}(u)+\int_{\Omega}\Gamma^{\frac{\alpha-2}{2-\beta}}|\nabla_{X} u|^{2}d\nu\leq
\left(\frac{(\beta+\alpha-4)(\alpha-2)}{2}+\epsilon\right)\int_{\Omega}
\Gamma^{\frac{\alpha-4}{2-\beta}} |\nabla_{X} \Gamma^{\frac{1}{2-\beta}}|^{2}
|u|^{2}\,d\nu\\+\frac{1}{4\epsilon}\int_{\Omega}\frac{\Gamma^{\frac{\alpha}{2-\beta}}}
{|\nabla_{X}\Gamma^{\frac{1}{2-\beta}}|^{2}}|\mathcal{L}u|^{2}d\nu.
\end{multline}
The already obtained inequality \eqref{LR2a} can be rewritten as
\begin{multline*}
\frac{16}{(\beta+\alpha-4)^{2}(\beta-\alpha)^{2}}\int_{\Omega}\frac{\Gamma^{\frac{\alpha}{2-\beta}}}
{|\nabla_{X}\Gamma^{\frac{1}{2-\beta}}|^{2}}
|\mathcal{L}u|^{2}d\nu-\frac{4}{(\beta+\alpha-4)(\beta-\alpha)}
\mathcal{C}(u)\\-\frac{4}{(\beta-\alpha)(\beta-2)}\int_{\partial\Omega}
\Gamma^{\frac{\alpha-2}{2-\beta}-1}|u|^{2}
\langle\widetilde{\nabla}\Gamma,
d\nu\rangle\geq \int_{\Omega}
\Gamma^{\frac{\alpha-4}{2-\beta}} |\nabla_{X} \Gamma^{\frac{1}{2-\beta}}|^{2}
|u|^{2}\,d\nu.
\end{multline*}
Combining it with \eqref{2R5a} we obtain
\begin{multline*}
\frac{1}{2}\mathcal{C}(u)+\left(\frac{(\beta+\alpha-4)(\alpha-2)}{2}+\epsilon\right)\frac{4}{(\beta+\alpha-4)(\beta-\alpha)}
\mathcal{C}(u)
\\
\left(\frac{(\beta+\alpha-4)(\alpha-2)}{2}+\epsilon\right)
\frac{4}{(\beta-\alpha)(\beta-2)}\int_{\partial\Omega}
\Gamma^{\frac{\alpha-2}{2-\beta}-1}|u|^{2}
\langle\widetilde{\nabla}\Gamma,
d\nu\rangle+\int_{\Omega}\Gamma^{\frac{\alpha-2}{2-\beta}}|\nabla_{X} u|^{2}d\nu\leq\\
\left(\frac{16\epsilon}{(\beta+\alpha-4)^{2}(\beta-\alpha)^{2}}+\frac{8(\alpha-2)}
{(\beta+\alpha-4)(\beta-\alpha)^{2}}
+\frac{1}{4\epsilon}\right)\int_{\Omega}\frac{\Gamma^{\frac{\alpha}{2-\beta}}}
{|\nabla_{X}\Gamma^{\frac{1}{2-\beta}}|^{2}}|\mathcal{L}u|^{2}d\nu.
\end{multline*}
Taking
$\epsilon=\frac{(\beta+\alpha-4)(\beta-\alpha)}{8}$ this implies
\begin{multline*}
\frac{(\beta+\alpha-4)(\beta-\alpha)}{4}\mathcal{C}(u)+\frac{(\beta+3\alpha-8)(\beta+\alpha-4)(\beta-\alpha)}{8(\beta-2)}\int_{\partial\Omega}
\Gamma^{\frac{\alpha-2}{2-\beta}-1}|u|^{2}
\langle\widetilde{\nabla}\Gamma,
d\nu\rangle\\+
\frac{(\beta-\alpha)^{2}}{4}\int_{\Omega}\Gamma^{\frac{\alpha-2}{2-\beta}}|\nabla_{X} u|^{2}d\nu\leq\int_{\Omega}\frac{\Gamma^{\frac{\alpha}{2-\beta}}}
{|\nabla_{X}\Gamma^{\frac{1}{2-\beta}}|^{2}}|\mathcal{L}u|^{2}d\nu.
\end{multline*}

Proof of \eqref{2LR2}. Inequality
\eqref{LR2} can be rewritten as
\begin{multline*}
\frac{16}{(\beta+\alpha-4)^{2}(\beta-\alpha)^{2}}\int_{\Omega}\frac{\Gamma^{\frac{\alpha}{2-\beta}}}
{|\nabla_{X}\Gamma^{\frac{1}{2-\beta}}|^{2}}
|\mathcal{L}u|^{2}d\nu-\frac{16}{(\beta+\alpha-4)^{2}(\beta-\alpha)^{2}}\mathcal{R}\\\geq \int_{\Omega}
\Gamma^{\frac{\alpha-4}{2-\beta}} |\nabla_{X} \Gamma^{\frac{1}{2-\beta}}|^{2}
|u|^{2}\,d\nu,
\end{multline*}
where
\begin{multline*}\mathcal{R}:=\frac{(\beta+\alpha-4)(\beta-\alpha)}{8}\int_{\Omega}\Gamma^{\frac{\alpha-4}{2-\beta}}
|\nabla_{X} \Gamma^{\frac{1}{2-\beta}}|^{2}\,
\left({\ln}\frac{R}{\Gamma^{\frac{1}{2-\beta}}}\right)^{-2}|u|^{2}d\nu\\+\frac{(\beta+\alpha-4)(\beta-\alpha)}
{4(\beta-2)}\int_{\partial\Omega}
\Gamma^{\frac{\alpha-2}{2-\beta}-1}\left({\ln}\frac{R}{\Gamma^{\frac{1}{2-\beta}}}\right)^{-1}|u|^{2}
\langle\widetilde{\nabla}\Gamma,
d\nu\rangle\\+
\frac{(\beta+\alpha-4)^{2}(\beta-\alpha)}{4(\beta-2)}\int_{\partial\Omega}
\Gamma^{\frac{\alpha-2}{2-\beta}-1}|u|^{2}
\langle\widetilde{\nabla}\Gamma,
d\nu\rangle+\frac{(\beta+\alpha-4)(\beta-\alpha)}{4}\mathcal{C}(u).
\end{multline*}
Combining it with \eqref{2R5a} we obtain
\begin{multline*}
\frac{1}{2}\mathcal{C}(u)+\left(\frac{(\beta+\alpha-4)(\alpha-2)}{2}+\epsilon\right)
\frac{16}{(\beta+\alpha-4)^{2}(\beta-\alpha)^{2}}\mathcal{R}+\int_{\Omega}\Gamma^{\frac{\alpha-2}{2-\beta}}|\nabla_{X} u|^{2}d\nu\leq\\
\left(\frac{16\epsilon}{(\beta+\alpha-4)^{2}(\beta-\alpha)^{2}}+\frac{8(\alpha-2)}
{(\beta+\alpha-4)(\beta-\alpha)^{2}}
+\frac{1}{4\epsilon}\right)\int_{\Omega}\frac{\Gamma^{\frac{\alpha}{2-\beta}}}
{|\nabla_{X}\Gamma^{\frac{1}{2-\beta}}|^{2}}|\mathcal{L}u|^{2}d\nu.
\end{multline*}
Taking
$\epsilon=\frac{(\beta+\alpha-4)(\beta-\alpha)}{8}$ we obtain

\begin{equation*}
\frac{(\beta-\alpha)^{2}}{8}\mathcal{C}(u)+\frac{\beta+3\alpha-8}{2(\beta+\alpha-4)}\mathcal{R}+
\frac{(\beta-\alpha)^{2}}{4}\int_{\Omega}\Gamma^{\frac{\alpha-2}{2-\beta}}|\nabla_{X} u|^{2}d\nu\leq\int_{\Omega}\frac{\Gamma^{\frac{\alpha}{2-\beta}}}
{|\nabla_{X}\Gamma^{\frac{1}{2-\beta}}|^{2}}|\mathcal{L}u|^{2}d\nu.
\end{equation*}
The proof is complete.
\end{proof}

\begin{rem}\label{strLiegroups}
In particular, for example for stratified Lie groups we obtain refinements compared to Kombe \cite{Kombe:Rellich-Carnot-2010}, with respect to the inclusion of boundary terms.
Taking $\beta=Q\geq 3$ the homogeneous dimension of the group $\mathbb{G},$ the sub-Laplacian $\mathcal{L}=\Delta_{\mathbb{G}}$, and
$\Gamma^{\frac{1}{2-\beta}}(x)=d(x)$ a quasi-distance (sometimes called the $\mathcal{L}$-gauge) on the group, if $u\in C_{0}^{\infty}(\Omega),$ we have $\mathcal{C}(u)=0,$ and
\eqref{LR2a} is reduced to
\begin{equation}
\qquad \qquad \int_{\Omega}\frac{d^{\alpha}}
{|\nabla_{X}d|^{2}}
|\Delta_{\mathbb{G}}u|^{2}d\nu\geq \frac{(Q+\alpha-4)^{2}(Q-\alpha)^{2}}{16}\int_{\Omega}
d^{\alpha-4} |\nabla_{X} d|^{2}
|u|^{2}\,d\nu,
\end{equation}
for  $Q>\alpha>4-Q,$ \eqref{LR2} is reduced to
\begin{multline*}
\qquad \qquad \int_{\Omega}\frac{d^{\alpha}}
{|\nabla_{X}d|^{2}}
|\Delta_{\mathbb{G}}u|^{2}d\nu\geq \frac{(Q+\alpha-4)^{2}(Q-\alpha)^{2}}{16}\int_{\Omega}
d^{\alpha-4} |\nabla_{X} d|^{2}
|u|^{2}\,d\nu\\+\frac{(Q+\alpha-4)(Q-\alpha)}{8}\int_{\Omega}d^{\alpha-4}
|\nabla_{X} d|^{2}\,
\left({\ln}\frac{R}{d}\right)^{-2}|u|^{2}d\nu
, \; Q>\alpha>4-Q,
\end{multline*}
and \eqref{2LR2a} is reduced to
\begin{equation*}
\qquad \qquad \int_{\Omega}\frac{d^{\alpha}}
{|\nabla_{X}d|^{2}}|\Delta_{\mathbb{G}}u|^{2}d\nu\geq
\frac{(Q-\alpha)^{2}}{4}\int_{\Omega}d^{\alpha-2}|\nabla_{X} u|^{2}d\nu,\; Q>\alpha>\frac{8-Q}{3}
\end{equation*}
and \eqref{2LR2} is reduced to
\begin{multline*}\label
\qquad \qquad \int_{\Omega}\frac{d^{\alpha}}
{|\nabla_{X}d|^{2}}
|\Delta_{\mathbb{G}}u|^{2}d\nu\geq \frac{(Q-\alpha)^{2}}{4}\int_{\Omega}d^{\alpha-2}|\nabla_{X} u|^{2}d\nu\\+\frac{(Q+3\alpha-8)(Q-\alpha)}{16}\int_{\Omega}d^{\alpha-4}
|\nabla_{X} d|^{2}\,
\left({\ln}\frac{R}{d}\right)^{-2}|u|^{2}d\nu,\; Q>\alpha>\frac{8-Q}{3}.
\end{multline*}
For unweighted versions (with $\alpha=0$) inequalities \eqref{LR2a}-\eqref{LR2} work under the condition $Q\geq 5$ which is usually appearing in Rellich inequalities, while 
\eqref{2LR2a}-\eqref{2LR2} work for homogeneous dimensions $Q\geq 9$.

\end{rem}
\section{Examples}
\label{Sec4}

Here we discuss the settings (E1), (E2) and (E3) from the introduction in more detail.
Concerning (E1), indeed, the sub-Laplacians on Carnot groups (stratified Lie groups) can serve as examples for which
Theorems \ref{LHardy}, \ref{LRellich} and \ref{2LRellich} hold.
In this setting, we note that a global result can be inferred from the local result by using the homogeneity of the
fundamental solution $\Gamma$.
We refer to \cite[Section 7]{Ruzhansky-Suragan:Carnot groups} for a detailed discussion of this subject.
We note that in this case of stratified groups (and, in fact, on general homogeneous Lie groups), the invariant vector fields $X_k$ always assume the form \eqref{Xk}, see e.g.
\cite[Section 3.1.5]{FR}.

Let us now consider the settings (E2) and (E3). Here, $\mathcal{L}=\sum_{k=1}^{N}X_{k}^{2}$ is the sum of squares of vector fields
on $\mathbb{R}^{n}$ in the form \eqref{Xk}, i.e.
$1\leq N\leq n,$ and
\begin{equation}\label{vfs}
X_{k}=\frac{\partial}{\partial x_{k}}+
\sum_{m=N+1}^{n}a_{k,m}(x)
\frac{\partial}{\partial x_{m}},\quad k=1,\ldots,N,
\end{equation}
where $a_{k,m}(x)$ are locally $C^{r,\alpha}$-regular the definition of which we now briefly recall.
Let $d_c(x,y)$ be the control distance associated to the vector fields $X_k$, i.e. the infimum of $T>0$ such that there is a piecewise continuous integral curve $\gamma$ of $X_1,\ldots,X_N$ such that $\gamma(0)=x$ and $\gamma(T)=y$.
The H\"older space $C^{\alpha}(\Omega)$ is then defined for $0<\alpha\leq 1$ as the space of all
functions $u$ for which there is $C>0$ such that
$$
|u(x)-u(y)|\leq C d_c^\alpha(x,y)
$$
holds for all $x,y\in\Omega$. Then, $u\in C^{1,\alpha}$ if $X_k u\in C^{\alpha}$ for all $k=1,\ldots, N$, and $C^{r,\alpha}$ are defined inductively.

The existence of a local fundamental solution $\Gamma_y$ for such operators in the setting (E2) was established in
\cite{M12}. While the positivity of $\Gamma_y$ does not seem to be explicitly stated there, it follows from general
results in \cite{FS86, S84}.
In the setting (E3), the assumption {\rm (${\rm A}_{y}^{+}$)} holds in view of
\cite[Theorem 4.8 and Theorem 5.9]{BBMP}.

So, we now check the validity of the property \eqref{astokes} for arbitrary domains with piecewise smooth boundary.

For this, we do not need to make any assumptions on the step to which the H\"ormander's commutator condition is satisfied, whether it is satisfied or not, or on the existence of fundamental solutions as in {\rm (${\rm A}_{y}$)}.  Thus, we formulate this property as a general statement which may be of interest on its own.
The smoothness assumption on $X_k$ can be reduced here, e.g. to $a_{k,m}\in C^1$.

\begin{prop}
\label{stokes}
Let $\Omega\subset\mathbb{R}^{n}$ be an open
bounded domain with a piecewise smooth boundary that has no self-intersections.
Let $X_k$, $k=1,\ldots,N$, be $C^1$ vector fields in the form \eqref{vfs}.
Let $f_{k}\in C^{1}(\Omega)\bigcap C(\overline{\Omega}),\,k=1,\ldots,N$.
Then for each $k=1,\ldots,N,$ we have
\begin{equation}\label{EQ:S1}
\int_{\Omega}X_{k}f_{k}d\nu=
\int_{\partial\Omega}f_{k} \langle X_{k},d\nu\rangle.
\end{equation}
Consequently, we also have the divergence-type formula
\begin{equation}\label{EQ:S2}
\int_{\Omega}\sum_{k=1}^{N}X_{k}f_{k}d\nu=
\int_{\partial\Omega}
\sum_{k=1}^{N} f_{k}\langle X_{k},d\nu\rangle.
\end{equation}
If $y\in\mathbb R^{n}$ is such that {\rm (${\rm A}_{y}$)} is satisfied, then we can also take
$f_{k}=vX_{k}\Gamma_{y}$ in formulae above, for all $v\in C^{1}(\Omega)\bigcap C(\overline{\Omega})$.
\end{prop}

The latter formula \eqref{EQ:S2} is exactly the one needed for the admissibility of a domain in Definition \ref{DEF:adomain}.

\begin{proof}[Proof of Proposition \ref{stokes}]
For any function $f$ we calculate the following differentiation formula

$$df=\sum_{k=1}^{N}\frac{\partial f}{\partial x_{k}}dx_{k}
+\sum_{m=N+1}^{n}
\frac{\partial f}{\partial x_{m}} dx_{m}
$$
$$
=\sum_{k=1}^{N} X_{k}f dx_{k}-\sum_{k=1}^{N}
\sum_{m=N+1}^{n}a_{k,m}(x)
\frac{\partial f}{\partial x_{m}} dx_{k}$$
$$+
\sum_{m=N+1}^{n}
\frac{\partial f}{\partial x_{m}} dx_{m}=\sum_{k=1}^{N} X_{k}f dx_{k}$$
$$+\sum_{m=N+1}^{n}\frac{\partial f}
{\partial x_{m}} (-\sum_{k=1}^{N}a_{k,m}(x)
 dx_{k}+dx_{m})$$
$$
=\sum_{k=1}^{N} X_{k}f dx_{k}+\sum_{m=N+1}^{n}\frac{\partial f}
{\partial x_{m}} \theta_{m},
$$
where
\begin{equation}\label{theta}
\theta_{m}=-\sum_{k=1}^{N}a_{k,m}(x)
 dx_{k}+dx_{m},\,\,m=N+1,\ldots,n.
\end{equation}
That is
\begin{equation}\label{df}
df=\sum_{k=1}^{N} X_{k}f dx_{k}+\sum_{m=N+1}^{n}\frac{\partial f}
{\partial x_{m}} \theta_{m}.
\end{equation}
It is simple to see that
 $$\langle X_{s}, dx_{j}\rangle=
 \frac{\partial}{\partial x_{s}}dx_{j}=\delta_{sj},\; 1\leq s\leq N,\; 1\leq j\leq n,$$
where $\delta_{sj}$ is the Kronecker delta, and
$$\langle X_{s}, \theta_{m} \rangle=
\left\langle\frac{\partial}{\partial x_{s}}+
\sum_{g=N+1}^{n}a_{s,g}(x)
\frac{\partial}{\partial x_{g}},
-\sum_{k=1}^{N}a_{k,m}(x)
 dx_{k}+dx_{m}
\right\rangle
$$
$$=-\sum_{k=1}^{N}\left(\frac{\partial}{\partial x_{s}}a_{k,m}(x)\right)
 dx_{k}-\sum_{k=1}^{N}a_{k,m}(x)
\frac{\partial}{\partial x_{s}}
 dx_{k}
+\frac{\partial}{\partial x_{s}}dx_{m}
$$
 $$
 -\sum_{k=1}^{N}\sum_{g=N+1}^{n}a_{s,g}
 (x)\left(\frac{\partial}{\partial x_{g}}a_{k,m}(x)\right)dx_{k}
 -\sum_{k=1}^{N}\sum_{g=N+1}^{n}a_{s,g}
 (x)a_{k,m}(x)\frac{\partial}{\partial x_{g}}dx_{k}
 $$
 $$+
\sum_{g=N+1}^{n}a_{s,g}(x)
\frac{\partial}{\partial x_{g}}dx_{m}
=-\sum_{k=1}^{N}\left(\frac{\partial}{\partial x_{s}}a_{k,m}(x)\right)
 dx_{k}
-\sum_{k=1}^{N}a_{k,m}(x)\delta_{sk}$$
 $$
 -\sum_{k=1}^{N}\sum_{g=N+1}^{n}a_{s,g}
 (x)\left(\frac{\partial}{\partial x_{g}}a_{k,m}(x)\right)dx_{k}
+ \sum_{g=N+1}^{n}a_{s,g}(x) \delta_{gm}$$
 $$
 =-\sum_{k=1}^{N}\sum_{g=N+1}^{n}a_{s,g}
 (x)\left(\frac{\partial}{\partial x_{g}}a_{k,m}(x)\right)dx_{k}
-\sum_{k=1}^{N}\left(\frac{\partial}{\partial x_{s}}a_{k,m}(x)\right)
 dx_{k}$$
$$
=-\sum_{k=1}^{N} \bigg[\sum_{g=N+1}^{n}a_{s,g}
 (x)\left(\frac{\partial}{\partial x_{g}}a_{k,m}(x)\right)
 +\frac{\partial}{\partial x_{s}}a_{k,m}(x)\bigg]
 dx_{k}.
$$
That is, we have
 $$\langle X_{s}, dx_{j}\rangle=\delta_{sj},$$
for $s=1,\ldots,N,\;j=1,\ldots,n,$ and
$$\langle X_{s}, \theta_{m} \rangle=
\sum_{k=1}^{N}\mathcal{C}_{k}(s,m) dx_{k},
$$
for $s=1,\ldots,N,\,\,m=N+1,\ldots,n$. Here
$$\mathcal{C}_{k}(s,m)=-\sum_{g=N+1}^{n}a_{s,g}(x)
\frac{\partial}{\partial x_{g}}a_{k,m}(x)-\frac{\partial}{\partial
x_{s}}a_{k,m}(x).$$ We have
$$d\nu:=d\nu(x)=\bigwedge_{j=1}^{N}dx_{j}
=
\bigwedge_{j=1}^{N}dx_{j}\bigwedge_{m=N+1}^{n}dx_{m}
=\bigwedge_{j=1}^{N}dx_{j}\bigwedge_{m=N+1}^{n}\theta_{m},$$
so
\begin{equation}\label{Xjdnu}
\langle X_{k}, d\nu(x)\rangle=\bigwedge_{j=1,j\neq k}^{N}dx_{j}
\bigwedge_{m=N+1}^{n}\theta_{m}.
\end{equation}
Therefore, by using Formula \eqref{df} we get
$$d(f_{s}\langle X_{s}, d\nu(x)\rangle)=df_{s}\wedge\langle X_{s}, d\nu(x)\rangle
=\sum_{k=1}^{N} X_{k}f_{s} dx_{k}\wedge\langle X_{s}, d\nu(x)\rangle$$$$+\sum_{m=N+1}^{n}\frac{\partial f_{s}}
{\partial x_{m}} \theta_{m}\wedge\langle X_{s}, d\nu(x)\rangle=
\sum_{k=1}^{N} X_{k}f_{s}dx_{k}\wedge\bigwedge_{j=1,j\neq k}^{N}dx_{j}
\bigwedge_{m=N+1}^{n}\theta_{m}$$
$$ +\sum_{m=N+1}^{n}\frac{\partial f_{s}}
{\partial x_{m}} \theta_{m}\wedge\bigwedge_{j=1,j\neq k}^{N}dx_{j}
\bigwedge_{m=N+1}^{n}\theta_{m},$$
here the first term is equal to $X_{s}f_{s}d\nu(x)$ and the second term is zero by the wedge
product rules,
that is,
\begin{equation}\label{stokesone}
d(\langle f_{s}X_{s}, d\nu(x)\rangle)=X_{s}f_{s}d\nu(x),\quad s=1,\ldots,N.
\end{equation}
Now using the Stokes theorem (see e.g. \cite[Theorem 26.3.1]{DFN}) we obtain \eqref{EQ:S1}.
Taking a sum over $k$ we also obtain \eqref{EQ:S2} for $f_{k}\in C^{1}(\Omega)\bigcap C(\overline{\Omega})$.

As in the classical case, the formula \eqref{EQ:S1} is still valid
for the fundamental solution of $\mathcal{L}$ since $\Gamma$ can
be estimated by a distance function associated to $\{X_{k}\}$ (see
e.g. \cite[Proposition 4.8]{M12}), or \cite{FS86,S84} for a more
general setting.
\end{proof}

For both cases (E2) and (E3) we present the following explicit example (see
\cite[Section 6]{BBMP}): In $\mathbb{R}^{3}$ let $N=2$ and let
$$X_{1}=\frac{\partial}{\partial x_{1}}+a(x)
\frac{\partial}{\partial x_{3}},$$
$$X_{2}=\frac{\partial}{\partial x_{2}}+b(x)
\frac{\partial}{\partial x_{3}},$$
with (non-smooth) coefficients
$$a(x)=x_{2}(1+|x_{2}|),\;b(x)=-x_{1}(1+|x_{1}|).$$
Then
$$[X_{1}, X_{2}]=-2(1+|x_{1}|+|x_{2}|)\frac{\partial}{\partial x_{3}}$$
and the sub-Laplacian is
$$\mathcal{L}=X_{1}^{2}+X_2^{2}.$$
The vector fields $X_{1},X_{2}$ are $C^{1,1}$ and satisfy H\"ormander's commutator
condition of step two, hence Assumptions of (E2) hold.
Replacing $|x_{1}|$, $|x_{2}|$ with $x_{1} |x_{1}|$, $x_{2} |x_{2}|$
we find $C^{2,1}$ vector fields, satisfying Assumptions of (E3).

Some other examples can be built from $\Delta_{\lambda}$-Laplacians, see e.g.
\cite{Kogoj-Sonner:Hardy-lambda-CVEE-2016}.




\end{document}